\documentclass[12pt,twoside,a4paper]{article}

\usepackage[active]{srcltx}
\usepackage{fancyhdr}
\usepackage{titlesec}
\usepackage{titletoc}
\usepackage{graphicx}

\usepackage{amsmath}
\usepackage{amstext}
\usepackage{amsbsy}
\usepackage{amssymb}
\usepackage{theorem}
\usepackage{dsfont}
\usepackage{times}
\usepackage[T1]{fontenc}
\usepackage{eucal}
\usepackage{mathrsfs}
\usepackage{url}
\newcommand{\ds}{\displaystyle}
\newcommand{\ts}{\textstyle}
\newcommand{\tint}{{\ts \int}}

\renewcommand{\mathbb}{\mathds}

\newcommand{\C}[1]{\mathcal{#1}}

\newcommand{\ov}[1]{\overline{#1}}
\newcommand{\wt}[1]{\widetilde{#1}}
\newcommand{\wh}[1]{\widehat{#1}}
\newcommand{\B}[1]{\mathds{#1}}

\DeclareMathOperator{\Var}{\mathbb{V}ar}

\numberwithin{equation}{section}
\newtheorem{thm}{Theorem}[section]
\newtheorem{prop}[thm]{Proposition}

\newtheorem{lemma}[thm]{Lemma}
\newtheorem{cor}[thm]{Corollary}

{\theorembodyfont{\rm} 

}
\theoremstyle{remark}

\theoremheaderfont{\sc}
\newenvironment{proof}{{\sc Proof.}}{\ $\square$}
\titleformat{\section}[block]{\sc\center}{\thetitle.}{5pt}{}[]
\titlespacing{\section}{0pt}{*4.5}{*3}
\titleformat{\subsection}[runin]{\sc}{\thetitle.}{5pt}{}[.]
\titlespacing{\subsection}{0pt}{*3}{*2}
\titleformat{\subsubsection}[runin]{\it}{\thetitle.}{5pt}{}[.]
\titlespacing{\subsubsection}{0pt}{*2}{*2}
\contentsmargin[30pt]{30pt}
\titlecontents{section}[15pt]{\vspace{6pt}\sc}
{\contentslabel[\thecontentslabel.]{15pt}}{\hspace*{-15pt}}
{\titlerule*[1pc]{.}\contentspage[\rm\thecontentspage]}
\titlecontents{subsection}[37pt]{\vspace{3pt}\small\sc}
{\contentslabel[\thecontentslabel.]{22pt}}{\hspace*{-22pt}}
{\titlerule*[1pc]{.}\contentspage[\normalsize\rm\thecontentspage]}
\titlecontents{subsubsection}[69pt]{\it}
{\contentslabel[\thecontentslabel.]{32pt}}{\hspace*{-32pt}}
{~\titlerule*[1pc]{.}\contentspage[\rm\thecontentspage]}

\newcommand{\thmref}[1]{\ref{#1} (page \pageref{#1})}
\newcommand{\myeq}[1]{{\rm (\ref{#1}, page \pageref{#1})}}
\allowdisplaybreaks

\DeclareMathOperator{\XX}{\text{\bf \textsf X}}

\begin{document}
\renewcommand{\sectionmark}[1]{\markboth{\thesection\ #1}{}}
\renewcommand{\subsectionmark}[1]{\markright{\thesubsection\ #1}}
\fancyhf{}
\fancyhead[RE]{\small\sc\nouppercase{\leftmark}}
\fancyhead[LO]{\small\sc\nouppercase{\rightmark}}
\fancyhead[LE,RO]{\thepage}
\fancyfoot[RO,LE]{\small\sc Olivier Catoni}
\fancyfoot[LO,RE]{\small\sc\today}
\renewcommand{\footruleskip}{1pt}
\renewcommand{\footrulewidth}{0.4pt}
\newcommand{\mypoint}{\makebox[1ex][r]{.\:\hspace*{1ex}}}
\addtolength{\footskip}{11pt}
\pagestyle{plain}
\begin{center}
{\bf Challenging the empirical mean and empirical variance: a deviation study}\\[12pt]
{\sc Olivier Catoni\footnote{CNRS --  UMR 8553,
D\'epartement de Math\'ematiques et Applications,
Ecole Normale Sup\'erieure, 45, rue d'Ulm, F75230 Paris 
cedex 05, and INRIA Paris-Rocquencourt -- CLASSIC team.}
}\\[12pt]
{\small \it \today }\\[12pt]
\end{center}


{\small
{\sc Abstract :} 
We present new M-estimators of the mean and variance of real valued random 
variables, based on PAC-Bayes bounds.
We analyze the non-asymptotic minimax properties of 
the deviations of those estimators for sample distributions having
either a bounded variance or a bounded variance and a bounded kurtosis. 
Under those weak hypotheses, allowing for heavy-tailed distributions,
we show that the worst case deviations of the empirical mean 
are suboptimal. We prove indeed that for 
any confidence level, there is some M-estimator 
whose deviations are of the same order as the 
deviations of the empirical mean of a Gaussian statistical 
sample, even when the statistical sample is instead heavy-tailed.
Experiments reveal that these new estimators perform 
even better than predicted by our bounds, showing deviation quantile functions uniformly
lower at all probability levels than the empirical 
mean for non-Gaussian sample distributions as 
simple as the mixture of two Gaussian measures. 
\\[1ex]
}
\pagestyle{fancy}

\newcommand\eqdef{\triangleq}
\newcommand\A{\mathcal{A}}
\newcommand\cB{\mathcal{B}}
\newcommand\cC{\Theta}
\newcommand\E{\mathbb{E}}
\newcommand\cE{\mathcal{E}}
\newcommand\cF{\mathcal{F}}
\newcommand\cH{\mathcal{H}}
\newcommand\I{\mathcal{I}}
\newcommand\J{\mathcal{J}}
\newcommand\cK{\mathcal{K}}
\newcommand\cL{\mathcal{L}}
\newcommand\M{\mathcal{M}}
\newcommand\N{\mathcal{N}}
\renewcommand\P{\mathbb{P}}
\newcommand\R{\mathbb{R}}
\renewcommand\S{\mathcal{S}}
\newcommand\bcR{B} 
\newcommand\cR{\tilde{B}} 
\newcommand\W{\mathcal{W}}
\newcommand\X{\mathcal{X}}
\newcommand\Y{\mathcal{Y}}
\newcommand\Z{\mathcal{Z}}

\newcommand\jyem{\em}

\newcommand\hdelta{\hat{\delta}}
\newcommand\hth{\hat{\theta}}
\newcommand\tth{\tilde{\theta}}
\newcommand\hlam{\hat{\lam}}
\newcommand\lamerm{\hlam^{\textnormal{(erm)}}}

\newcommand\tf{\tilde{f}}
\newcommand\tlam{\tilde{\lam}}

\newcommand\be{\beta}
\newcommand\ga{\gamma}
\newcommand\kap{\kappa}
\newcommand\lam{\lambda}

\newcommand\sint{{\textstyle \int}}
\newcommand\inth{\sint \pi(d\theta)}
\newcommand\inthj{\sint \pij(d\theta)}
\newcommand\inthp{\sint \pi(d\theta')}

\newcommand\gas{\ga^*}

\newcommand\hpi{\hat{\pi}}
\newcommand\pis{\pi^*}
\newcommand\tpi{\tilde{\pi}}

\newcommand\ovth{\ov{\theta}}
\newcommand\wth{\wt{\theta}}
\newcommand\wthj{\wt{\theta}_j}

\newcommand\eps{\varepsilon}
\newcommand\logeps{\log(\eps^{-1})}
\newcommand\leps{\log^2(\eps^{-1})}

\newcommand{\bi}{\begin{itemize}}
\newcommand{\ei}{\end{itemize}}

\newcommand\ovR{\ov{R}}

\newcommand\hhpi{\hat{\hat{\pi}}}
\newcommand\wwth{\wt{\wt{\theta}}}
\newcommand\pia{\pi^{(1)}}
\newcommand\pib{\pi^{(2)}}
\newcommand\pij{\pi^{(j)}}
\newcommand\tpia{\tilde{\pi}^{(1)}}
\newcommand\tpib{\tilde{\pi}^{(2)}}
\newcommand\tpij{\tilde{\pi}^{(j)}}
\newcommand\wtha{\wt{\theta}_1}
\newcommand\wthb{\wt{\theta}_2}
\newcommand\wths{\wt{\theta}_s}
\newcommand\wtht{\wt{\theta}_t}

\newcommand\cmin{c_{\min}}
\newcommand\cmax{c_{\max}}

\newcommand\ra{\rightarrow}

\newcommand\hatt{\hat{t}}
\newcommand\argmax{\textnormal{argmax}}
\newcommand\argmin{\textnormal{argmin}}

\newcommand\diag{\textnormal{Diag}}

\newcommand\Fg{\cF^\#}

\newcommand\vp{\varphi}
\newcommand\tvp{\tilde{\vp}}
\newcommand\tphi{\tilde{\phi}}
\renewcommand\th{\theta}

\newcommand\vsp{\vspace{1cm}}
\newcommand\lhs{\text{l.h.s.}}
\newcommand\rhs{\text{r.h.s.}}

\newcommand\expe[2]{\undc{\E}{#1\sim#2}}
\newcommand\expec[2]{\undc{\E}{#1\sim#2}}
\newcommand\expecc[2]{\E_{#1}}
\newcommand\expecd[2]{\E_{#2}}

\newcommand\hC{\hat{C}}
\newcommand\hf{\hat{f}}
\newcommand\hrho{\hat{\rho}}

\newcommand\br{\bar{r}}
\newcommand\chr{\check{r}}
\newcommand\bR{\bar{R}}

\newcommand\lan{\langle}
\newcommand\ran{\rangle}

\newcommand\logepsg{\log(|\Cg|\eps^{-1})}

\newcommand\Pemp{\hat{\P}}

\newcommand\fracl[2]{{(#1)}/{#2}}
\newcommand\fracc[2]{{#1}/{#2}}
\newcommand\fracr[2]{{#1}/{(#2)}}
\newcommand\fracb[2]{{(#1)}/{(#2)}}

\newcommand\hfproj{\hf^{\textnormal{(proj)}}}
\newcommand\thproj{\hat{\th}^{\textnormal{(proj)}}}

\newcommand\hfols{\hf^{\textnormal{(ols)}}}
\newcommand\thfols{\tilde{f}^{\textnormal{(ols)}}}
\newcommand\thols{\hat{\th}^{\textnormal{(ols)}}}
\newcommand\therm{\hat{\th}^{\textnormal{(erm)}}}
\newcommand\hferm{\hf^{\textnormal{(erm)}}}
\newcommand\zols{\zeta^{\textnormal{(ols)}}}

\newcommand\thrid{\tilde{\th}} 
\newcommand\frid{\tilde{f}} 
\newcommand\freg{f^{\textnormal{(reg)}}}
\newcommand\hfrlam{\hf^{\textnormal{(ridge)}}} 
\newcommand\hfllam{\hf^{\textnormal{(lasso)}}} 

\newcommand\flin{f^*_{\textnormal{lin}}}
\newcommand\thlin{\th^{\textnormal{(lin)}}}
\newcommand\Flin{\mathcal{F}_{\textnormal{lin}}}

\renewcommand\Phi{\XX}
\newcommand\demi{\frac{1}{2}}
\newcommand\demic{\fracc{1}{2}}

\newcommand\substa[2]{\substack{#1\\#2}}
\newcommand\substac[2]{{#1\,;\,#2}}
\newcommand\tpsi{\tilde{\psi}}
\newcommand\tzeta{\tilde{\zeta}}
\newcommand\ta{\tilde{a}}
\newcommand\chis{\chi_\sigma}
\newcommand\tchi{\tilde{\chi}}
\newcommand\tchis{\tchi_\sigma}
\newcommand\psis{\psi_\sigma}

\newcommand\tA{\tilde{A}}
\newcommand\tL{\tilde{L}}

\newcommand\hL{\hat{L}}
\newcommand\hcE{\hat{\cE}}
\newcommand\hDe{\hat{\cE}}
\newcommand\cEb{\cE^{\sharp}}
\newcommand\La{L^{\flat}}
\newcommand\cEa{\cE^{\flat}}
\newcommand\Lb{L^{\sharp}}

\newcommand\Pa{P^{\flat}}
\newcommand\Pb{P^{\sharp}}

\newcommand\ela{\tilde{\ell}}
\newcommand\hpig{\hpi^{\textnormal{(Gibbs)}}}

\newcommand\sigmb{\phi}
\newcommand{\tR}{\tilde{\cR}}
\newcommand{\logdeps}{\log(4d\eps^{-1})}
\newcommand{\logddeps}{\log(2d^2\eps^{-1})}

\section{Introduction}
This paper is devoted to the estimation of the mean and possibly also 
of the variance of a real random variable
from an independent identically distributed sample. 
While the most traditional way to deal with this question is to focus 
on the mean square error of estimators, 
we will instead focus on their deviations.
Deviations are related to the estimation of confidence 
intervals which are of importance in many situations. 
While the empirical mean 
has an optimal minimax mean square error among all mean 
estimators in all models including 
Gaussian distributions, its deviations 
tell a different story. 
Indeed, as far as the mean square error is concerned, Gaussian distributions
represent already the worst case, so that in the framework 
of a minimax mean least 
square analysis, no need is felt to improve estimators for 
non-Gaussian sample distributions. 
On the contrary, the deviations 
of estimators, and especially of the empirical mean,
are worse for non-Gaussian samples than 
for Gaussian ones. Thus 
a deviation analysis will point out 
possible improvements of the empirical mean 
estimator more successfully. 
It was nonetheless quite unexpected for us, and will undoubtedly 
also be for some of our readers, that the empirical 
mean could be improved, under such a weak hypothesis as the existence 
of a finite variance, and that this has remained 
unnoticed until 
now. One of the reasons may 
be that the weaknesses of the empirical mean 
disappear if we let the sample distribution be 
fixed and the sample size go to infinity. 
This does not mean however that a substantial 
improvement is not possible, nor that it is 
only concerned with specific sample sizes 
or weird worst case distributions : 
in the end of this paper, we 
will present experiments made on quite simple 
sample distributions, consisting in the mixture
of two to three Gaussian measures, showing that 
more than twenty five percent can be gained on the widths of
confidence intervals, for realistic sample sizes ranging from 100 to 2000. 
We think that, beyond the technicalities involved here, 
this exemplifies more broadly the pitfalls of asymptotic
studies in statistics and should be quite thought provoking 
about the notions of optimality commonly used to assess the 
performance of estimators. 

Our deviation study will use 
two kinds of tools: M-estimators to truncate observations 
and PAC-Bayesian theorems to combine estimates on the same sample 
without using a split scheme \cite{McA98,McA99,McA01,Cat05,Aud03b,Alq08}.

Its general conclusion is that, whereas the deviations of the 
empirical mean estimate may increase a lot when the 
sample distribution is far from being Gaussian, those of some new M-estimators
will not. The improvement is the best for heavy-tailed distributions, 
as the worst case analysis performed to prove lower bounds will show.  
The improvement also increases as the confidence level at which 
deviations are computed increases.

Similar conclusions can be drawn in the case of least square regression
with random design \cite{AuCat10a, AuCat10b}. Discovering that using truncated estimators permits 
to get rid of sub Gaussian tail assumptions was the spur to study 
the simpler case of mean estimation for its own sake.
Restricting the subject in this way (which is of course
a huge restriction compared with least square regression) makes it 
possible to propose simpler dedicated estimators 
and to push their analysis further. It will indeed be 
possible here to obtain mathematical proofs for 
numerically significant non-asymptotic bounds.

The weakest hypothesis we will consider is the existence of 
a finite but unknown variance. In our M-estimators, 
adapting the truncation level 
depends on the value of the variance. 
However, this adaptation can be done without actually 
knowing the variance, through Lepski's approach.

Computing an observable confidence interval, 
on the other hand, requires more information. 
The simplest case is when the variance is known, or 
at least lower than a known bound.
If it is not so, another possibility is to assume 
that the kurtosis is known, or lower than a known bound. 
Introducing the kurtosis is natural to our approach: 
in order to calibrate the truncation level for the 
estimation of the mean, we need to know the variance,
and in the same way, in order to calibrate the truncation 
level for the estimation of the variance, we need to use 
the variance of the variance, which is provided by the 
kurtosis as a function of the variance itself.

In order to assess the quality of the results, we prove corresponding
lower bounds for the best estimator confronted to the worst possible 
sample distribution, 
following the minimax approach. 
We also compute lower bounds for the deviations of the 
empirical mean estimator when the sample distribution is the worst 
possible. 
These latter bounds show the improvement that can be brought 
by M-estimators over the more traditional 
empirical mean. We plot the numerical values 
of these upper and lower bounds against the confidence
level for typical finite sample sizes to show the gap between them.

The reader may wonder why we only consider the following extreme
models, the narrow Gaussian model and the broad models 
\begin{align}
\label{eq1.1}
\C{A}_{v_{\max}} & = \bigl\{ \B{P} \in \C{M}_+^1(\B{R}) : 
\Var_{\B{P}} \leq v_{\max} \bigr\}, \\
\label{eq1.2.0}
\text{and } 
\C{B}_{v_{\max}, \kappa_{\max}} & = \bigl\{ \B{P} \in \C{A}_{v_{\max}} : 
\kappa_{\B{P}} \leq \kappa_{\max} \bigr\},
\end{align}
where $\Var_{\B{P}}$ is the variance of $\B{P}$, 
$\kappa_{\B{P}}$ its kurtosis, and $\C{M}_+^1(\B{R})$ is 
the set of probability measures (positive measures of mass $1$) 
on the real line equipped with the Borel sigma-algebra.

The reason is that, the minimax bounds obtained in these 
broad models being close to the ones obtained in the Gaussian 
model, introducing intermediate models would not change the 
order of magnitude of the bounds.

Let us end this introduction by advocating the value of confidence
bounds, stressing more particularly the case of high confidence levels,
since this is the situation where truncation brings the most 
decisive improvements. 

One situation of interest which we will 
not comment further is when the estimated parameter is 
critical and making a big 
mistake on its value, even with a small probability, is 
unacceptable.  

Another scenario to be encountered in statistical learning 
is the case when 
lots of estimates are to be computed and compared in the course 
of some decision making. 
Let us imagine, for instance, that some parameter
$\theta \in \Theta$ is to be tuned in order to optimize the expected 
loss $\B{E} \bigl[ f_{\theta}(X) \bigr]$ of some 
family  
of loss functions $\bigl\{ f_{\theta} 
: \theta \in \Theta \bigr\}$ computed on some random input $X$. Let us 
consider a split sample scheme where two i.i.d. samples $(X_1, \dots, 
X_s) \overset{\text{def}}{=} X_1^s$ and $(X_{s+1}, \dots, X_{s+n}) 
\overset{\text{def}}{=} X_{s+1}^{s+n}$ are used, one to build some 
estimators $\wh{\theta}_k(X_1^s)$ of $\argmin_{\theta \in \Theta_k} \B{E} 
\bigl[ f_{\theta}(X) \bigr]$ in subsets $\Theta_k$, $k=1, \dots, K$
of $\Theta$, and the other to test those estimators and keep hopefully
the best. This is a very common model selection situation. 
One can think for instance of the choice of a basis to expand 
some regression function. If $K$ is large, estimates of 
$\B{E} \bigl[ f_{\wh{\theta}_k(X_1^s)}(X_{s+1}) \bigr]$ will be required
for a lot of values of $k$. In order to keep safe from over-fitting,
very high confidence levels will be required if the resulting
confidence level is to be computed through a union bound (because
no special structure of the problem can be used to do better).
Namely, a confidence level of $1 - \epsilon$ on the final result
of the optimization on the test sample will require a confidence
level of $1 - \epsilon/K$ for each mean estimate on the test sample. 
Even if $\epsilon$ is not very small (like, say, $5/100$), 
$\epsilon/K$ may be very small. For instance, if 10 parameters
are to be selected among a set of 100, this gives $K = 
{100 \choose 10} \simeq 1.7 \cdot 10^{13}$. 
In practice, except in some special situations 
where fast algorithms exist, a heuristic scheme will be used 
to compute only a limited number of estimators $\wh{\theta}_k$. 
An example of heuristics is to 
add greedily parameters one at a time, choosing
at each step the one with the best estimated performance increase (in our example, this requires to compute $1000$ estimators instead of 
${100 \choose 10}$). Nonetheless,
asserting the quality of the resulting choice 
requires a union 
bound on the whole set of possible outcomes of the 
data driven heuristic, and therefore calls for  
very high confidence levels for each estimate of the mean performance 
$\B{E} \bigl[ f_{\wh{\theta}_k (X_1^s)}(X_{s+1}) \bigr]$ 
on the test set. 

The question we are studying in this paper
should not be confused with robust statistics \cite{Huber,Huber2}. 
The most fundamental difference is that we are interested in 
estimating the mean of the sample distribution. 
In robust statistics, it is assumed 
that the sample distribution is in the neighbourhood 
of some known parametric model. This gives the possibility 
to replace the mean by some other location 
parameter, which as a rule 
will not be equal to the mean when the shape of the distribution 
is not constrained (and in particular is not assumed 
to be symmetric).  

Other differences are that our point of view is 
non-asymptotic and that we study the deviations 
of estimators whereas robust statistics is focussed 
on their asymptotic mean square error. 

Although we end up defining M-estimators 
with the help of influence functions, like in robust 
statistics, we use a truncation level depending on 
the sample size, whereas in robust statistics, 
the truncation level depends on the amount of contamination. 
Also, we truncate at much higher levels (that is 
we eliminate less outliers) that what would be advisable
for instance in the case of a contaminated Gaussian
statistical sample. 
Thus, although we have some tools in common with 
robust statistics, we use them differently 
to achieve a different purpose.

Adaptive estimation of a location parameter \cite{Stone,Beran,Bickel} is
another setting where the empirical mean can be replaced by more efficient
estimators. However, the setting studied by these authors is quite different
from ours. The main difference, here again, is that the location parameter 
is assumed to be the center of symmetry of the sample distribution, 
a fact that is used to tailor location estimators based on a 
symmetrized density estimator.
Another difference is that in these papers, the estimators are built with
asymptotic properties in view, such as asymptotic normality with optimal
variance and asymptotic robustness. These properties, although desirable, 
give no information on the non-asymptotic deviations we are
studying here, and therefore do not provide as we do non-asymptotic confidence intervals.

\section{Some new M-estimators}

Let $(Y_i)_{i=1}^n$ be an i.i.d. sample drawn from some unknown probability 
distribution $\B{P}$ on the real line $\B{R}$ equipped with the
Borel $\sigma$-algebra $\C{B}$. Let $Y$ be independent from $(Y_i)_{i=1}^n$ 
with the same marginal distribution $\B{P}$. 
Assuming that $Y \in \bold{L}_2(\B{P})$, let $m$ be the mean of $Y$ and let $v$ be its variance: 
$$ 
\B{E}(Y) = m \quad \text{ and } \quad \B{E}[ (Y - m)^2 ] = v.
$$

Let us consider some non-decreasing\footnote{We would like to thank one of the 
anonymous referees 
of an early version of this paper for pointing out the 
benefits that could be drawn from the use of a non-decreasing 
influence function in this section.} influence function $\psi : \B{R} 
\rightarrow \B{R}$ such that
\begin{equation}
\label{eqFund}
-\log \bigl( 1 - x + x^2/2 \bigr) 
\leq \psi(x) \leq \log \bigl( 1 + x + x^2/2 \bigr).
\end{equation}
The {\em widest} possible choice of $\psi$ compatible with 
these inequalities is 
\begin{equation}
\label{eq1.2w}
\psi(x) = \begin{cases}
\log \bigl(1 + x +x^2/2 \bigr), & x \geq 0,\\
- \log \bigl( 1 - x + x^2/2 \bigr), & x \leq 0,
\end{cases}
\end{equation}
whereas the {\em narrowest} possible choice is  
\begin{equation}
\label{eq1.2}
\psi(x) = \begin{cases}
\log(2), & x \geq 1, \\
- \log \bigl( 1 - x + x^2/2 \bigr), & 0 \leq x \leq 1, \\
\log \bigl( 1 + x + x^2 /2 \bigr), & -1 \leq x \leq 0, \\
- \log(2), & x \leq -1.
\end{cases}
\end{equation}
Although $\psi$ is not the derivative of some explicit
error function, we will use it in the same way, so 
that it can be considered as an influence function.

Indeed, $\alpha$ being some positive real parameter
to be chosen later, 
we will build our estimator $\wh{\theta}_{\alpha}$ of the mean $m$
as the solution of the equation 
$$
\sum_{i=1}^n \psi \bigl[ \alpha (Y_i - \wh{\theta}_{\alpha} ) \bigr]  
= 0.
$$
(When the narrow choice of $\psi$ defined by equation \eqref{eq1.2} 
is made, the above equation may have more 
than one solution, in which case any of them can be used 
to define $\wh{\theta}_{\alpha}$.)\\
\mbox{} \hfill \includegraphics{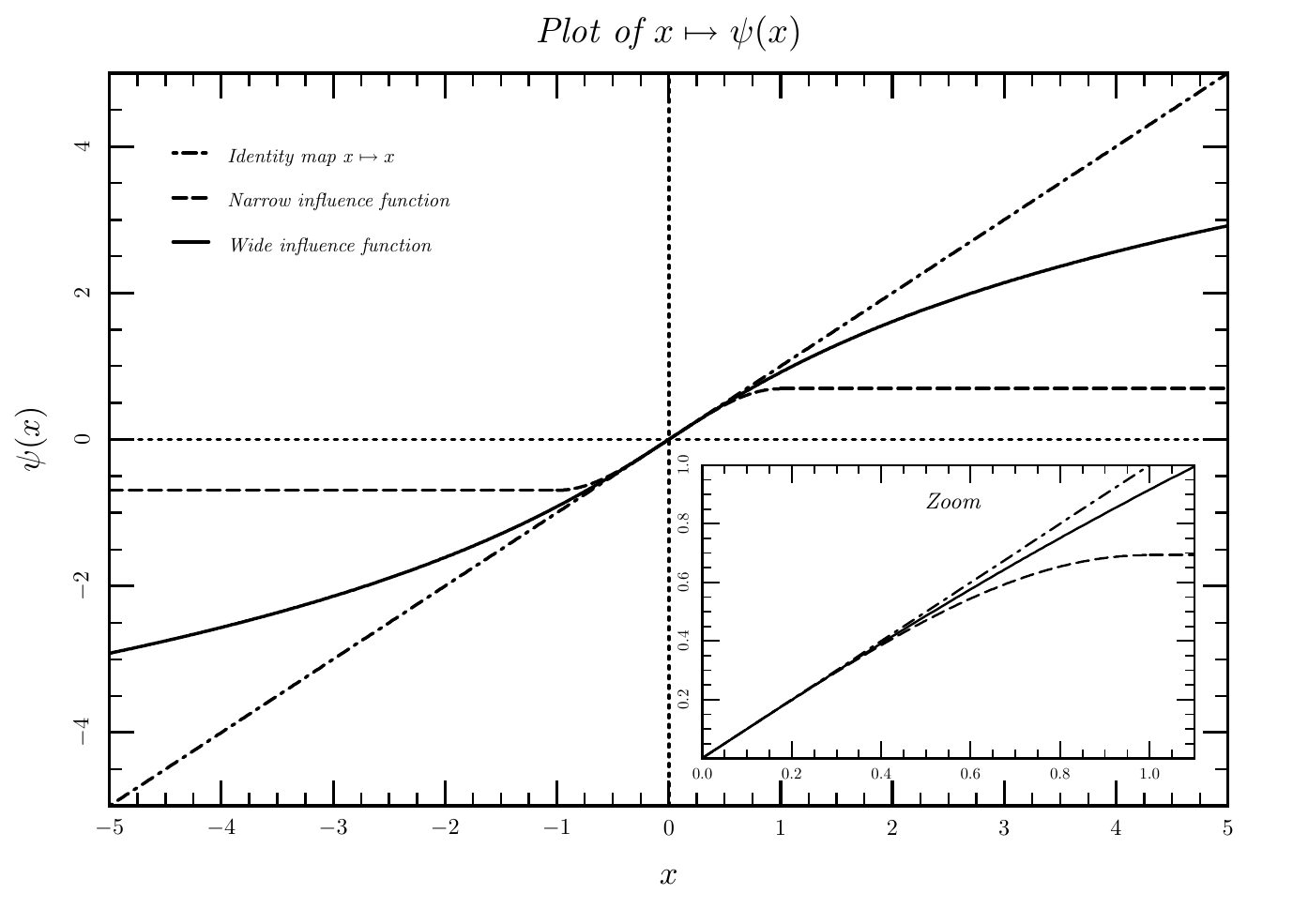}   
\hfill \mbox{}\\[-4ex]

The widest choice of $\psi$ is the one making $\wh{\theta}_{\alpha}$ 
the closest to the empirical mean. For this reason it may 
be preferred if our aim is to stabilize the empirical mean
by making the smallest possible change, which could be 
justified by the fact that the empirical mean is optimal
in the case when the sample $(Y_i)_{i=1}^n$ is Gaussian.

Our analysis of $\wh{\theta}_{\alpha}$ will rely on the following 
exponential moment inequalities, from which deviation bounds 
will follow. Let us introduce the quantity
$$
r(\theta) = \frac{1}{\alpha n} \sum_{i=1}^n \psi \bigl[ 
\alpha (Y_i - \theta) \bigr], \qquad \theta \in \B{R}. 
$$
\begin{prop}
\begin{multline*}
\B{E} \Bigl\{ \exp \bigl[ \alpha n r(\theta) \bigr] \Bigr\}
\leq  
\Bigr\{ 
1 + \alpha(m-\theta) + \frac{\alpha^2}{2} \bigl[ v + (m-\theta)^2\bigr]  
\Bigr\}^n \\ \leq \exp \Bigl\{ n \alpha(m-\theta) + \frac{n\alpha^2}{2} 
\bigl[ v + (m-\theta)^2 \bigr] \Bigr\}.
\end{multline*}
In the same way
\begin{multline*}
\B{E} \Bigl\{ \exp \bigl[ - \alpha n r(\theta) \bigr] \Bigr\}
\leq  
\Bigr\{ 
1 - \alpha(m-\theta) + \frac{\alpha^2}{2} \bigl[ v + (m-\theta)^2\bigr]  
\Bigr\}^n \\ \leq \exp \Bigl\{ - n \alpha(m-\theta) + \frac{n\alpha^2}{2} 
\bigl[ v + (m-\theta)^2 \bigr] \Bigr\}.
\end{multline*}
\end{prop}
The proof of this proposition is an obvious consequence of 
inequalities \eqref{eqFund} and of the fact that the sample
is assumed to be i.i.d.
It justifies the special choice of influence 
function we made. If we had taken for $\psi$ 
the identity function, and thus for $\wh{\theta}_{\alpha}$ the empirical 
mean, the exponential moments of $r(\theta)$ would have existed only 
in the case when the random variable $Y$ itself has exponential moments. 
In order to bound $\wh{\theta}_{\alpha}$, we will find two non-random 
values $\theta_-$ and $\theta_+$ of the parameter such that with large
probability $r(\theta_-) > 0 > r(\theta_+)$, which will imply that 
$\theta_- < \wh{\theta}_{\alpha} < \theta_+$, since $r(\wh{\theta}_{\alpha}) = 0 $
by construction and $\theta \mapsto r(\theta)$ is non-increasing. 

\begin{prop}
\label{prop1.2}
The values of the parameters $\alpha \in \B{R}_+$ 
and $\epsilon \in )0,1($ being set, 
let us define for any $\theta \in \B{R}$ the bounds
\begin{align*}
B_+(\theta) & = m - \theta + \frac{\alpha}{2} 
\bigl[ v + (m-\theta)^2 \bigr] + \frac{\log\bigl( \epsilon^{-1} \bigr)}{
n \alpha}, \\
B_-(\theta) & = m - \theta - \frac{\alpha}{2} 
\bigl[ v + (m-\theta)^2 \bigr] - \frac{\log\bigl( \epsilon^{-1} \bigr)}{
n \alpha}.
\end{align*}
They satisfy  
$$
\B{P} \bigl[ r(\theta) < B_+(\theta) \bigr] \geq 1 - \epsilon, \quad  
\B{P} \bigl[ r(\theta) > B_-(\theta) \bigr] \geq 1 - \epsilon.
$$
\end{prop}
The proof of this proposition is also straightforward: it is a mere 
consequence of Chebyshev's inequality and of the previous proposition.
Let us assume that 
$$
\alpha^2 v + \frac{2 \log\bigl(\epsilon^{-1}\bigr)}{n}
\leq 1.
$$
Let $\theta_+$ be the smallest solution of the quadratic equation 
$B_+(\theta_+) = 0$ and let $\theta_-$ be the largest solution 
of the equation $B_-(\theta_-) = 0$. 

\begin{lemma}
\begin{align*}
\theta_+ & = 
m + \left( \frac{\alpha v}{2} 
+ \frac{\log \bigl( \epsilon^{-1} \bigr)}{ \alpha n} \right) 
\left( \frac{1}{2} + \frac{1}{2} \sqrt{1 - \alpha^2 v - \frac{2 \log\bigl(\epsilon^{-1}\bigr)}{n}}\, \right)^{-1} 
\\ & \leq 
m + \left( \frac{\alpha v}{2} 
+ \frac{\log \bigl( \epsilon^{-1} \bigr)}{ \alpha n} \right) 
\left( 1 - \frac{\alpha^2 v}{2} - \frac{\log \bigl( \epsilon^{-1} 
\bigr)}{n} \right)^{-1}, \\  
\theta_- & 
 = m - \left( \frac{\alpha v}{2} 
+ \frac{\log \bigl( \epsilon^{-1} \bigr)}{ \alpha n} \right) 
\left( \frac{1}{2} + \frac{1}{2} \sqrt{1 - \alpha^2 v - 
\frac{2 \log\bigl(\epsilon^{-1}\bigr)}{n}}\, \right)^{-1} \\
& \geq m - \left( \frac{\alpha v}{2} 
+ \frac{\log \bigl( \epsilon^{-1} \bigr)}{ \alpha n} \right) 
\left( 1 - \frac{\alpha^2 v}{2} - \frac{\log \bigl( \epsilon^{-1} 
\bigr)}{n} \right)^{-1}.
\end{align*}
Moreover, since the map $\theta \mapsto r(\theta)$ is non-increasing, 
with probability at least $1 - 2 \epsilon$, 
$$
\theta_- < \wh{\theta}_{\alpha} < \theta_+.
$$
\end{lemma}
The proof of this lemma is also an obvious consequence 
of the previous proposition and of the definitions of 
$\theta_+$ and $\theta_-$. 
Optimizing the choice of $\alpha$ provides
\begin{prop}
\label{prop1.4}
Let us assume that $ n > 2 \log(\epsilon^{-1}) $
and let us consider 
\begin{align*}
\eta & =
 \sqrt{ \frac{2 v \log(\epsilon^{-1})}{\ds n \left( 1 - \frac{2 \log(\epsilon^{-1})}{n} \right)}} &  
\text{ and } & & \alpha & = \sqrt{\frac{2 \log \bigl( \epsilon^{-1} \bigr)}{n (v + \eta^2)}}.
\end{align*}
In this case $\theta_+ = m + \eta$ and $\theta_- = m - \eta$, 
so that with probability at least $1 - 2 \epsilon$, 
$$
\lvert m - \wh{\theta}_{\alpha} \rvert 
\leq \eta =  
 \sqrt{ \frac{2 v \log(\epsilon^{-1})}{\ds n \left( 1 - \frac{2 \log(\epsilon^{-1})}{n} \right)}}. 
$$

In the same way, if we want to make a choice of $\alpha$ independent
from $\epsilon$, we can choose
$$
\alpha = \sqrt{\frac{2}{n v}},
$$
and assume that
$$
n > 2 \bigl[1 + \log(\epsilon^{-1})\bigr].
$$
In this latter case, with probability at least $1 - 2 \epsilon$, 
$$
\lvert \wh{\theta}_{\alpha} - m \rvert \leq 
\frac{1 + \log \bigl( \epsilon^{-1} \bigr)}{\ds 
\frac{1}{2} + \frac{1}{2} \sqrt{1  - \frac{2[1 + \log ( \epsilon^{-1} )]}{n}}} 
\sqrt{\frac{v}{2n}} 
\leq \frac{1 + \log\bigl(\epsilon^{-1}\bigr)}{\ds 1 - 
\frac{1 + \log(\epsilon^{-1})}{n}} \sqrt{\frac{v}{2n}}.
$$
\end{prop}

In the following plots, we compare the bounds on the deviations
of our M-estimator $\wh{\theta}_{\alpha}$ with the deviations 
of the empirical mean $\ds M = \frac{1}{n} \sum_{i=1}^n Y_i $
when the sample distribution is Gaussian and when it belongs to the 
model $\C{A}_{1}$ defined in the introduction by equation \myeq{eq1.1}. 
(The bounds for the empirical mean will be explained and proved in 
subsequent sections.)

More precisely, the deviation upper bounds for our estimator $\wh{\theta}_{\alpha}$ 
for the worst sample distribution
in $\C{A}_1$, the model defined by \myeq{eq1.1}, 
are compared with
the exact deviations of the empirical
mean $M$ of a Gaussian sample.  
This is the minimax bound at all confidence levels
in the Gaussian model, as will be proved
later on. Consequently, the deviations of  
our estimator cannot be smaller for the worst 
sample distribution in $\C{A}_1$, which contains
Gaussian distributions. 
We see on the plots that starting 
from $\epsilon = 0.1$ (corresponding to a $80\%$ 
confidence level), 
our upper bound is close to being minimax, 
not only in $\C{A}_1$, 
but also in the small Gaussian sub-model. 
This shows that the deviations of our estimator are
close to reach the minimax bound in any intermediate model containing 
Gaussian distributions and contained in $\C{A}_1$. 

Our estimator is also compared with the deviations 
of the empirical mean for the worst distribution in $\C{A}_1$, 
(to be established later). In particular 
the lower bound proves that there are sample distributions 
in $\C{A}_1$ for which the deviations of the empirical 
mean are far from being optimal, showing the need 
to introduce a new estimator to correct this.

In the first plot, we chose a sample size 
$n = 100$ and plotted the deviations against the confidence level 
(or rather against $\epsilon$, the confidence level 
being $1 - 2 \epsilon$).

\noindent \mbox{} \hfill \raisebox{-0.8cm}[9.2cm][0cm]{\includegraphics{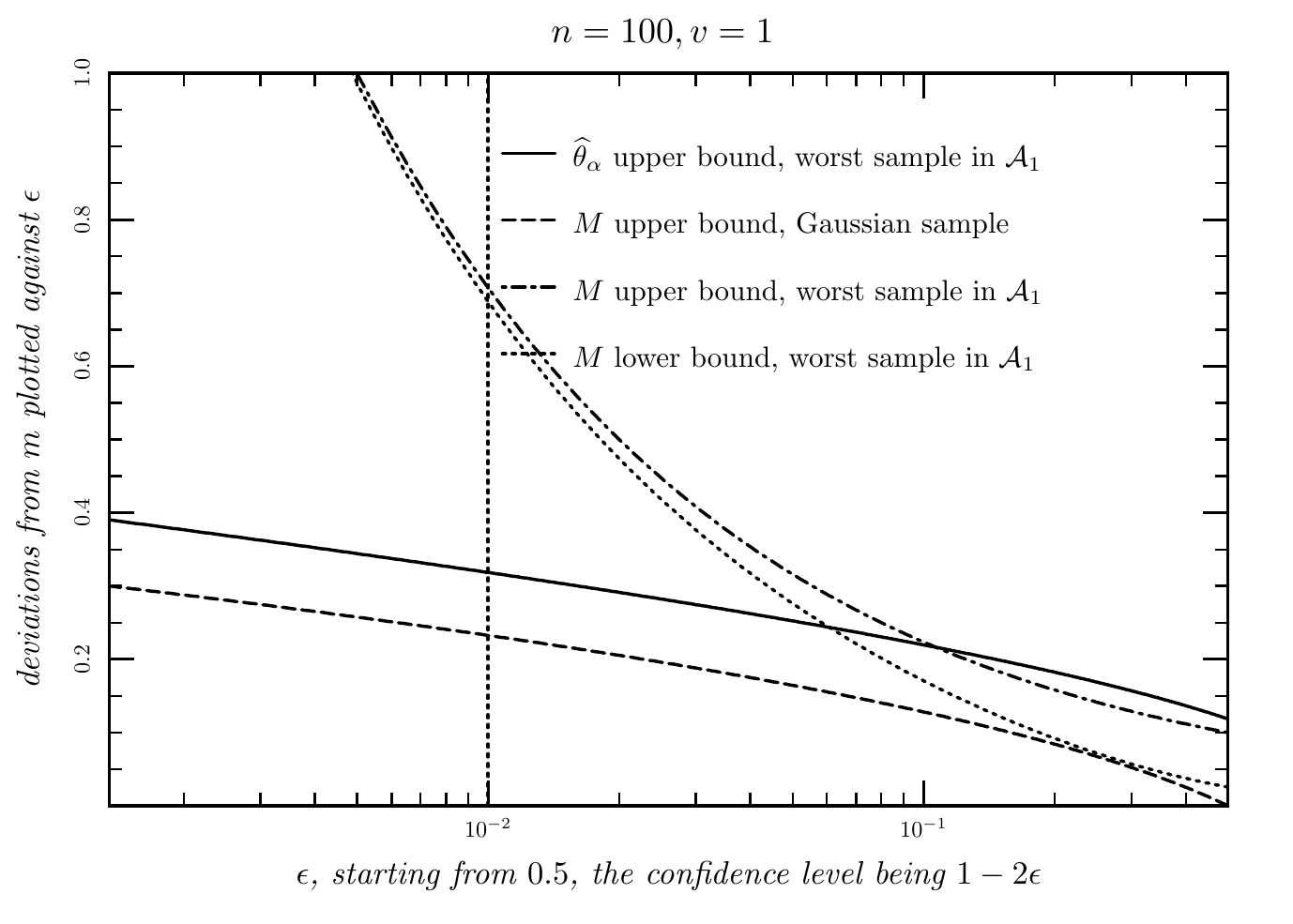}}
\hfill \mbox{}

As shown on the second plot, showing a wider range of 
$\epsilon$ values, our bound stays close to the Gaussian 
bound up to very high confidence levels (up to $\epsilon = 10^{-9}$
and more).
On the other hand, it already outperforms the empirical mean 
by a factor larger than two at confidence level $98 \%$ (that is 
for $\epsilon = 10^{-2}$).

\noindent \mbox{} \hfill \raisebox{-0.8cm}[9.2cm][0cm]{\includegraphics{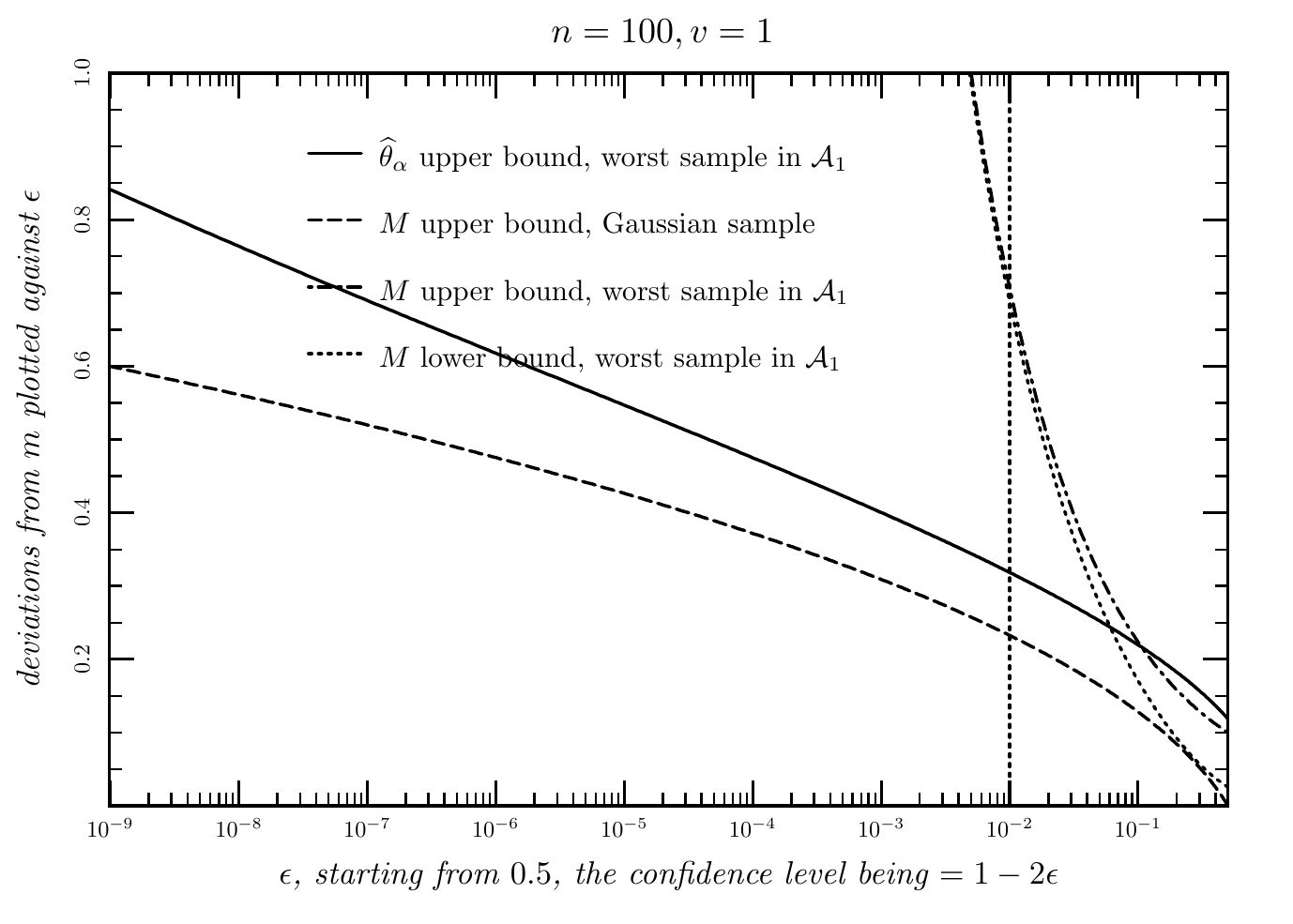}} 
\hfill \mbox{}

When we increase the sample size to $n = 500$,
the performance of our M-estimator is even closer to optimal. 

\noindent \mbox{} \hfill \raisebox{-0.8cm}[9cm][0cm]{\includegraphics{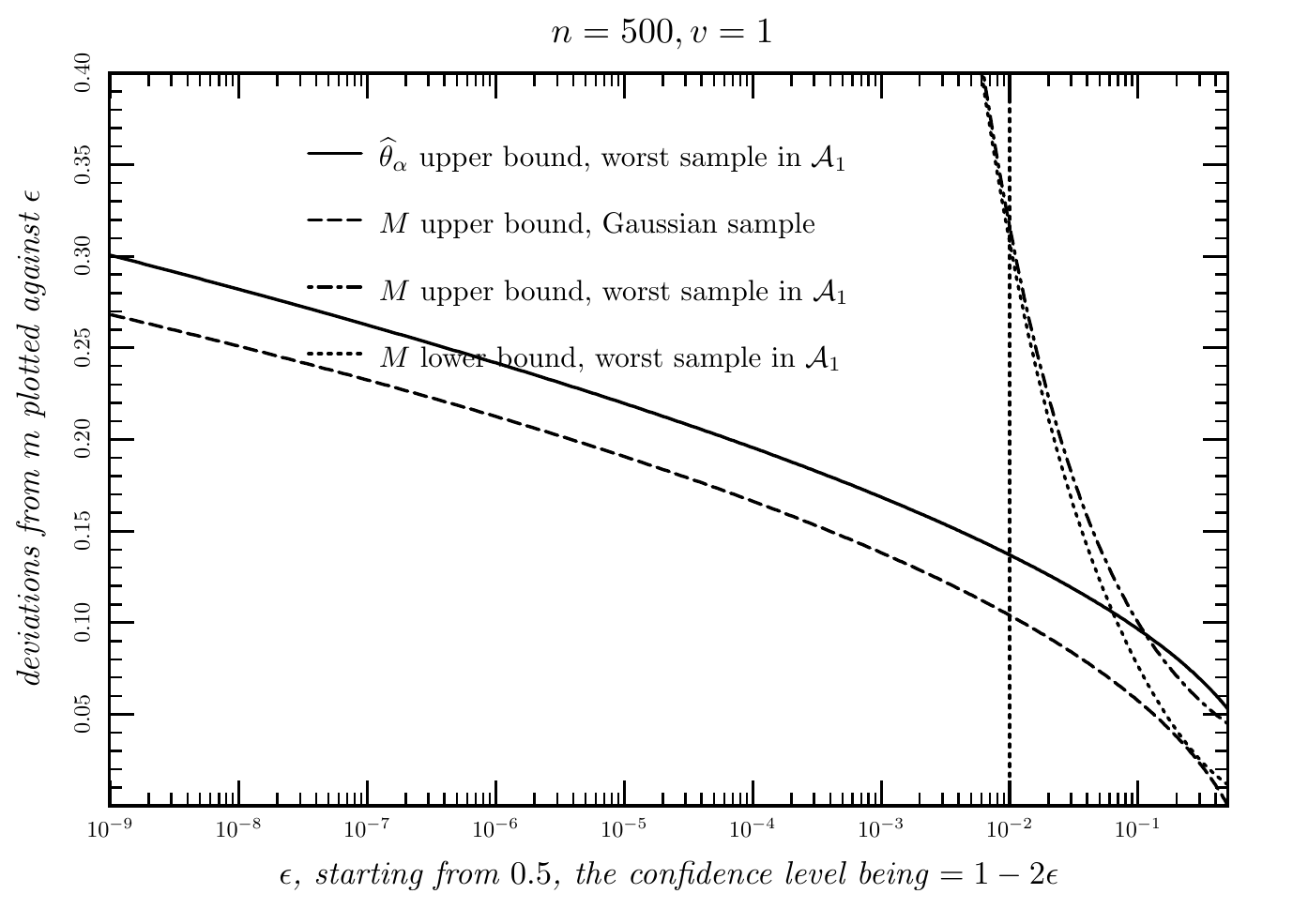}}   
\hfill \mbox{}

\section{Adapting to an unknown variance}

In this section, we will use Lepski's 
renowned adaptation method \cite{Lepski2} 
when nothing is known, except that the variance 
is finite. Under so uncertain, (but unfortunately
so frequent) circumstances, it is impossible to
provide any {\em observable} confidence intervals,
but it is still possible to define an adaptive estimator
and to bound its deviations by unobservable bounds 
(depending on the unknown variance).
To understand the subject of this section, 
one should keep in mind that {\em adapting} 
to the variance is a weaker requirement than {\em estimating} the variance : 
estimating the variance at any predictable rate 
would require more prior information (bearing for instance
on higher moments of the sample distribution). 

The idea of Lepski's method is powerful and simple : consider
a sequence of confidence intervals obtained by assuming that the variance
is bounded by a sequence of bounds $v_k$ and pick up as an estimator
the middle of the smallest interval intersecting all the larger 
ones. For this to be legitimate, we need all the confidence regions
for which the variance bound is valid to hold together, which 
is performed using a union bound.

Let us describe this idea more precisely. Let $\wh{\theta}(v_{\max})$ 
be some estimator of the mean depending on some assumed variance bound 
$v_{\max}$, as the ones described in the beginning of this paper. 
Let $\delta(v_{\max}, \epsilon) \in \B{R}_+ \cup \{ + \infty \}$  
be some deviation bound\footnote{$\C{A}_{v_{\max}}$ 
is defined by \myeq{eq1.1}} proved in $\C{A}_{v_{\max}}$: 
namely let us assume that for any sample distribution 
in $\C{A}_{v_{\max}}$, with probability at least $1 - 2 \epsilon$, 
$$
\lvert m -  \wh{\theta}(v_{\max}) \rvert \leq \delta(v_{\max}, \epsilon).
$$
Let us also decide by convention that 
$\delta(v_{\max}, 0) = + \infty$.

Let $\nu \in \C{M}_+^1(\B{R}_+)$ be some coding atomic probability 
measure on the positive real line, which will serve to take 
a union bound on a (countable) set of possible values of $v_{\max}$.

We can choose for instance for $\nu$ the following coding distribution : 
expressing $v_{\max}$ by comparison with some reference value 
$V$, 
$$
v_{\max} = V 2^s \sum_{k=0}^{d} c_k 2^{-k}, \quad
s \in \B{Z}, d \in \B{N}, (c_k)_{k=0}^d \in \{0,1\}^{d+1},
c_0 = c_d = 1,
$$
we set 
$$
\nu(v_{\max}) = \frac{2^{-2(d-1)}}{5(\lvert s \rvert + 2) (\rvert s \rvert + 3)} 
$$
and otherwise we set $\nu(v_{\max}) = 0$. 
It is easy to see that this defines a probability distribution on
$\B{R}_+$ (supported by dyadic numbers scaled by the factor V). 
It is clear that, as far as possible, the reference value $V$ should
be chosen as close as possible to the true variance $v$.  

Another possibility is to set for some parameters $V \in \B{R}$, 
$\rho > 1$ and $s \in \B{N}$, 
\begin{equation}
\label{eq3.1.0}
\nu(V \rho^{2k}) = \frac{1}{2s+1}, \quad k \in \B{Z}, \lvert k \rvert \leq s.
\end{equation}

Let us consider for any $v_{\max}$ such that $\delta(v_{\max}, \epsilon \nu(v_{\max}))
< + \infty$ the confidence interval 
$$
I(v_{\max}) = \wh{\theta}(v_{\max}) + \delta \bigl[ v_{\max}, \epsilon \nu(v_{\max}) 
\bigr] \times (-1,1).
$$
Let us put $I(v_{\max}) = \B{R}$ when $\delta(v_{\max}, \epsilon \nu(v_{\max})) = + 
\infty$. 

Let us consider the non-decreasing family of closed intervals
$$
J(v_1) = \bigcap \Bigl\{ I(v_{\max}) : v_{\max} \geq v_1 \Bigr\}, \qquad v_1 \in \B{R}_+.
$$
(In this definition, we can restrict the intersection to the support 
of $\nu$, since otherwise $I(v_{\max}) = \B{R}$.)
A union bound shows immediately that
with probability at least $1 - 2 \epsilon$, $m \in J(v)$, 
implying as a consequence that $J(v) \neq \varnothing$.

\begin{prop}
Since $v_1 \mapsto J(v_1)$ is a non-decreasing family of closed intervals, 
the intersection 
$$
\bigcap \Bigl\{ J(v_1) : v_1 \in \B{R}_+, J(v_1) \neq \varnothing \Bigr\}
$$
is a non-empty closed interval, and we can therefore pick up an adaptive 
estimator $\wt{\theta}$ belonging to it, choosing for instance 
the middle of this interval.

With probability at least $1 - 2 \epsilon$, $ m \in J(v)$, which implies
that $J(v) \neq \varnothing$, and therefore that $\wt{\theta} 
\in J(v)$. 

Thus with probability at least $1 - 2 \epsilon$
$$
\lvert m - \wt{\theta} \rvert \leq \lvert J(v) \rvert 
\leq 2 \inf_{v_{\max} > v}  \delta(v_{\max}, \epsilon \nu(v_{\max})). 
$$ 
If the confidence bound $\delta(v_{\max}, \epsilon)$ is 
homogeneous, in the sense that 
$$
\delta(v_{\max}, \epsilon) = B(\epsilon) \sqrt{v_{\max}}, 
$$
as it is the case in Proposition \thmref{prop1.4} 
with 
$$
B(\epsilon) =  
\sqrt{\frac{2 \log(\epsilon^{-1})}{\ds n \left(1 - \frac{ 2 \log(\epsilon^{-1})}{n} \right)}}
$$
then with probability at least 
$1 - 2 \epsilon$, 
$$
\lvert m - \wt{\theta} \rvert \leq 2 \inf_{v_{\max} > v}  
B\bigl[ \epsilon \nu(v_{\max}) \bigr] \sqrt{v_{\max}}.
$$
Thus in the case when $\nu$ is defined by equation 
\myeq{eq3.1.0} and 
$\lvert \log(v/V) \rvert \leq 2 s \log(\rho)$,
with probability at least $1 - 2 \epsilon$
$$
\lvert m - \wt{\theta} \rvert \leq 2 \rho B \left(
\frac{\epsilon}{2s+1} \right) \sqrt{v}. 
$$
\end{prop}
Let us see what happens for a sample size of $n = 500$, when 
we assume that $\lvert\log(v/V)\rvert \leq 2 \log(100)$ and 
we take $\rho = 1.05$.\\
\mbox{} \hfill \includegraphics{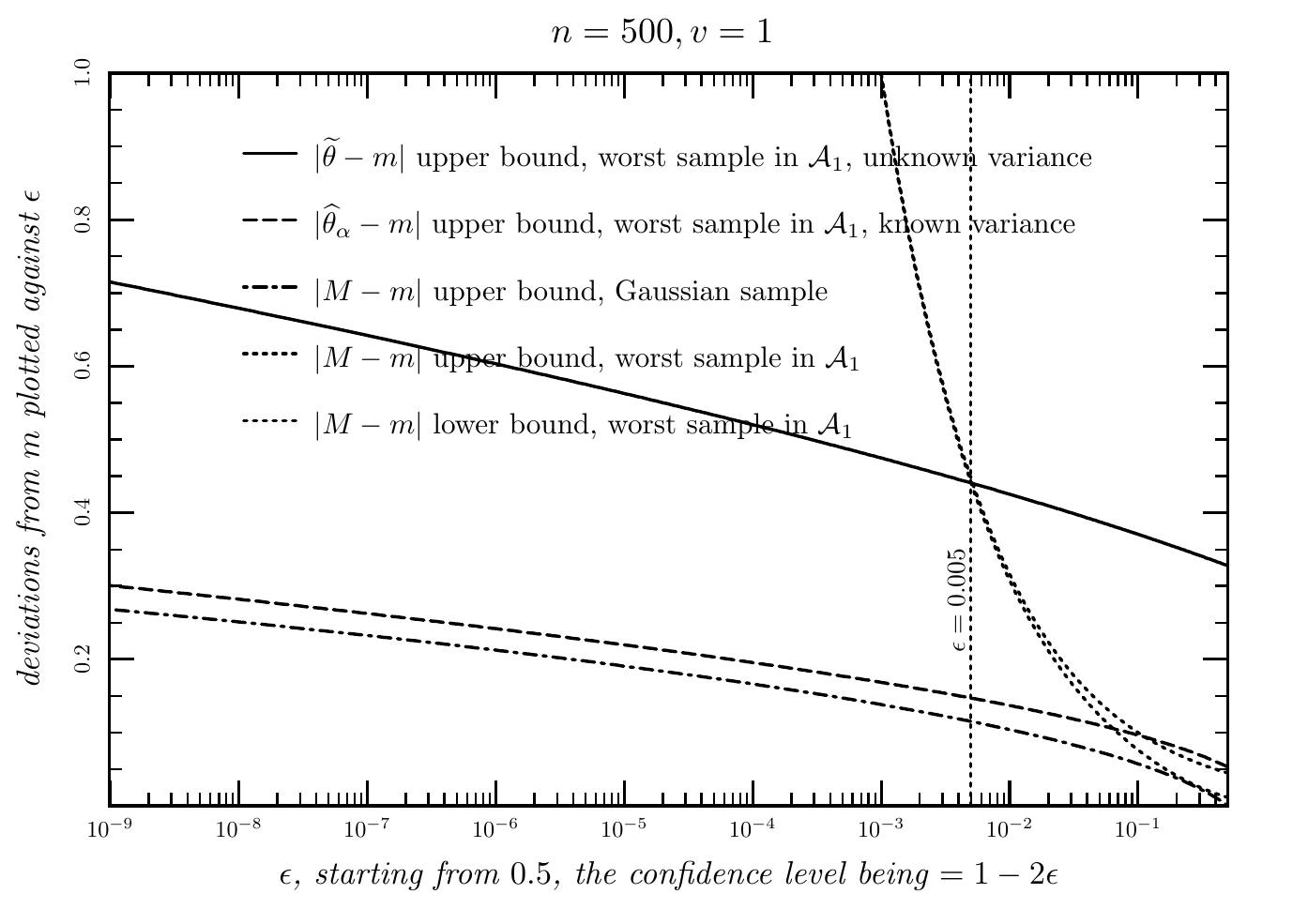}   
\hfill \mbox{}\\[-2ex]
This plot shows that, for a sample of size $n=500$, 
 there are sample distributions 
with a finite variance for which the deviations of 
the empirical mean blow up for confidence levels higher
than $99\%$, where as the deviations of our adaptive
estimator remain under control, even at confidence levels
as high as $1 - 10^{-9}$. 

The conclusion is that if our aim is to minimize the 
worst estimation error over 100 statistical experiments or more
and we have no information on the standard deviation except that 
it is in some range of the kind $(1/100, 100)$ (which 
is pretty huge and could be increased even more if 
desired), 
then the performance of the empirical mean estimator
for the worst sample distribution breaks down but thresholding 
very large outliers as $\wt{\theta}$ does can cure this problem.

\section{Mean and variance estimates depending on the kurtosis}

Situations where the variance is unknown are likely to happen.
We have seen in the previous section how to adapt 
to an unknown variance. The price to pay is  
a loss of a factor two in the deviation bound, and the fact that 
it is no longer observable. 

Here we will make hypotheses under which it is possible to 
estimate both the variance and the mean, and to obtain 
an observable confidence interval, without loosing a factor 
two as in the previous section. 
Making use of the kurtosis parameter is the most natural way 
to achieve these goals in the framework of our approach. 
This is what we are going to do here.

\subsection{Some variance estimate depending on the kurtosis}
In this section, we are going to consider an alternative 
to the unbiased usual variance estimate 
\begin{align}
\label{eq4.1bis}
\wh{V} & =  \frac{1}{n-1} 
\sum_{i=1}^n \biggl( Y_i - \frac{1}{n} \sum_{j=1}^n Y_j \biggr)^2 \\ 
& = \frac{1}{n(n-1)} \sum_{1 \leq i < j \leq n} \bigl( Y_i - Y_j \bigr)^2.
\nonumber
\end{align}
We will assume that the fourth moment $\B{E}(Y^4)$ is finite 
and that some upper bound is known for the kurtosis
$$
\kappa = \frac{\B{E} \bigl[ (Y-m)^4 \bigr]}{v^2}. 
$$
Our aim will be, as before with the mean, to define an estimate 
with better deviations than $\wh{V}$.
We will use $\kappa$ in the following computations, but when 
only an upper bound is known, 
$\kappa$ can be replaced with this upper bound in the definition 
of the estimator and the estimates of its performance.

Let us write $n = pq + r$, 
with $0 \leq r < p$, 
and let $\ds \{1, \dots, n \} = \bigsqcup_{\ell=1}^q I_\ell$
be the partition of the $n$ first intergers defined by 
$$
I_\ell 
= 
\begin{cases}
\{ i \in \B{N} ; p(\ell-1) < i \leq p \ell \}, & 1 \leq \ell < q,\\
\{ i \in \B{N} ; p(\ell-1) < i \leq n \}, & \ell = q.
\end{cases}
$$ 
We will develop some kind of block threshold scheme, introducing 
\begin{align*}
Q_{\delta}(\beta) & = \frac{1}{q}  \sum_{\ell = 1 }^q \psi \biggl( \frac{1}{\lvert I_{\ell} \rvert ( 
\lvert I_{\ell} \rvert - 1)} 
\sum_{\stackrel{i,j \in I_\ell}{i < j}} \bigl[ \beta (Y_i - Y_j)^2 - 2 \delta \bigr] \biggr) \\
& = \frac{1}{q} \sum_{\ell=1}^q \psi \Biggl[ \frac{\beta}{\lvert 
I_\ell \rvert - 1} \sum_{i \in I_\ell} \biggl( Y_i - 
\frac{1}{\lvert I_{\ell} \rvert} \sum_{j \in I_\ell} 
Y_j \biggr)^2 - \delta \Biggr]. 
\end{align*}
where $\psi$ is a non-decreasing influence function satisfying 
\myeq{eqFund}.

If $\psi$ were replaced with the identity, we would have $\B{E} 
\bigl[ Q_{\delta}(\beta) \bigr] = \beta v - \delta$. The idea is to solve 
$Q_{\delta}(\wh{\beta}) = 0$ in $\wh{\beta}$ and to estimate $v$ by 
$\delta / \wh{\beta}$. Anyhow, for technical reasons, we will adopt
a slightly different definition for $\wh{\beta}$ as well as for the 
estimate of $v$, as we will show now. Let us first get some deviation
inequalities for $Q_{\delta}(\beta)$, derived as usual from 
exponential bounds. It is straightforward to see that
\begin{multline*}
\B{E} \Bigl\{ \exp \bigl[ q \, Q(\beta) \bigr] \Bigr\}  \leq  \prod_{\ell=1}^q 
\Biggl\{ 1 + ( \beta v - \delta) \\ + \frac{1}{2} 
\biggl[ (\beta v-\delta)^2 + \frac{\beta^2}{\lvert I_{\ell} \rvert^2 (\lvert I_{\ell} \rvert - 1)^2} 
\sum_{\substack{i < j \in I_\ell \\ s < t \in I_{\ell}}} \B{E} \Bigl\{ \bigl[ (Y_i - Y_j)^2 
- 2v \bigr] \bigl[(Y_s - Y_{t})^2 - 2 v \bigr]  \Bigr\}  \biggr]  \Biggr\}.
\end{multline*}
We can now compute for any $i \neq j$
\begin{multline*}
\B{E} \Bigl\{ \bigl[ (Y_i - Y_j)^2 - 2 v \bigr]^2 \Bigr\} = 
\B{E} \bigl[ (Y_i - Y_j)^4 \bigr] - 4 v^2 \\ = 2 \kappa v^2 + 6 v^2 - 4 v^2 
= 2 (\kappa + 1) v^2,
\end{multline*}
and for any distinct values of $i$, $j$ and $s$,
\begin{multline*}
\B{E} \Bigl\{ \bigl[ (Y_i - Y_s)^2 - 2 v \bigr] \bigl[ (Y_j - Y_s)^2 
- 2 v\bigr] \Bigr\}  
= \B{E} \bigl[(Y_i - Y_s)^2(Y_j-Y_s)^2 \bigr] - 4 v^2 
\\ \shoveleft{\quad = \B{E} \Bigl\{ \bigl[ (Y_i - m)^2 + 2 (Y_i - m)(Y_s - m) 
+ (Y_s - m)^2 \bigr]} \\ \shoveright{\times \bigl[ (Y_j - m)^2 + 2 (Y_j - m)(Y_s - m) 
+ (Y_s - m)^2 \bigr] \Bigr\} - 4 v^2} \\ 
\shoveleft{\quad = \B{E} \bigl[ (Y_s - m)^4 \bigr] 
+ \B{E} \Bigl\{ (Y_s - m)^2 \bigl[ (Y_i - m)^2 + (Y_j - m)^2 \bigr] \Bigr\} }
\\ + \B{E} \bigl[ (Y_i - m)^2 (Y_j - m)^2 \bigr] - 4 v^2 = (\kappa - 1) v^2.
\end{multline*}

Thus 
\begin{multline*}
\sum_{\substack{i<j \in I_{\ell}\\ s < t \in I_{\ell}}} \B{E} \Bigl\{ \bigl[ (Y_i - Y_j)^2 
- 2v \bigr] \bigl[ (Y_s - Y_t)^2 - 2 v \bigr] \Bigr\} \\ = \lvert I_\ell \rvert ( \lvert I_\ell \rvert - 1) 
(\kappa+1) v^2 + \lvert I_{\ell} \rvert(\lvert I_{\ell} \rvert - 1) ( \lvert I_{\ell} \rvert - 2 )(\kappa -1) v^2 
\\ =   
\lvert I_{\ell} \rvert(\lvert I_{\ell} \rvert - 1)^2  \biggl[ (\kappa -1) 
+ \frac{2}{\lvert I_{\ell} \rvert - 1} \biggr]  v^2.
\end{multline*}
It shows that
\begin{multline*}
\B{E} \Bigl\{ \exp \bigl[ q Q(\beta) \bigr] \Bigr\} \leq  \biggl\{ 1 + ( \beta v-\delta) 
\\ + \frac{1}{2} ( \beta v-\delta)^2 + \frac{\beta^2 v^2}{2 p}
\biggl[ \kappa - 1 + \frac{2}{p-1} \biggr] 
\biggr\}^q.
\end{multline*}
In the same way
\begin{multline*}
\B{E} \Bigl\{ \exp \bigl[ - q Q(\beta) \bigr] \Bigr\} \leq  
\biggl\{ 1 - (\beta v - \delta) \\ + \frac{1}{2} ( \beta v - \delta)^2
+ \frac{\beta^2 v^2}{2p} \biggl[ 
\kappa -1 + \frac{2}{p-1} 
\biggr] \biggr\}^q.
\end{multline*}
Let $\ds \chi =  \kappa -1 + \frac{2}{p-1} \underset{p \rightarrow 
\infty}{\simeq} \kappa -1$. For any given $\beta_1, \beta_2 \in \B{R}_+$,  
with probability at least $1 - 2 \epsilon_1$, 
\begin{align*}
Q(\beta_1) & < \beta_1 v - \delta 
+ \frac{1}{2} (\beta_1 v - \delta)^2
+ \frac{\chi \beta_1^2 v^2 }{2p} + \frac{\log(\epsilon_1^{-1})}{q}, \\
Q(\beta_2) & > \beta_2 v - \delta - \frac{1}{2} (\beta_2 v - \delta )^2 
- \frac{\chi \beta_2^2 v^2}{2 p} - \frac{\log(\epsilon_1^{-1})}{q}.
\end{align*} 
Let us define, for some parameter $y \in \B{R}$,  $\wh{\beta}$ as  
\begin{align*}
Q(\wh{\beta}) & = - y, &
\text{and let us choose} &&
y & = \frac{\chi \delta^2}{2 p} + 
\frac{\log(\epsilon_1^{-1})}{q}, & \text{and} && 
\beta_2 & = \frac{\delta}{v},  
\end{align*}
so that $Q(\beta_2) >  - y$. 
Let us put $\xi = \delta - \beta_1 v$ and let us choose $\xi$ such 
that $Q(\beta_1) < - y$. This implies that $\xi$ is solution of
\begin{align*}
\frac{1+\zeta}{2} \xi^2 
- (1 + \zeta \delta ) \xi + 2 y & \leq 0 & 
\text{where} && \zeta & = \frac{\chi}{p}.
\end{align*}
Provided that $(1 + \zeta \delta)^2 \geq 4(1+\zeta)y$, 
the smallest solution of this equation is
$$
\xi = \frac{4 y}{1 + \zeta \delta + \sqrt{ (1 + \zeta \delta)^2 - 4 (1+\zeta) y}}.
$$
With these parameters, with probability at least $1 - 2 \epsilon_1$, 
$Q\bigl[ (\delta - \xi)/v \bigr] < - y < Q(\delta /v)$, implying that 
$$
\frac{\delta - \xi}{v} \leq \wh{\beta} \leq \frac{\delta}{v}.
$$
Thus putting
$$
\wh{v} = \frac{\sqrt{\delta(\delta - \xi)}}{\wh{\beta}}, 
$$
we get 
$$
\biggl( 1 - \frac{\xi}{\delta} \biggr)^{1/2} \leq \frac{v}{\wh{v}} \leq \biggl( 
1 - \frac{\xi}{\delta} \biggr)^{-1/2}.
$$
In order to minimize $\ds \frac{\xi}{\delta} \simeq \frac{2y}{\delta}$, we are led to 
take $\ds \delta = \sqrt{\frac{2 p \log(\epsilon_1^{-1})}{\chi q}}$.
We get 
\begin{align*}
y & = \frac{2 \log(\epsilon_1^{-1})}{q}, &
\frac{y}{\delta} & = \sqrt{\frac{2 \chi \log(\epsilon_1^{-1})}{n-r}},\\ 
\zeta \delta & = \frac{y}{\delta}, & 
\text{ and } \quad (1 + \zeta) y & = \frac{2 \log(\epsilon_1^{-1})}{q} + \frac{2 \chi \log(\epsilon_1^{-1})}{n-r}.
\end{align*}
Thus the condition becomes 
\begin{equation}
\label{eq4.1}
q \geq 8 \log(\epsilon_1^{-1}) \biggl( 1 + \frac{\chi}{p} \biggr) \Biggl( 
1 + \sqrt{ \frac{2 \chi \log(\epsilon_1^{-1})}{n-r}} \; \Biggr)^{-2}.
\end{equation}
\begin{prop}
\label{prop4.1}
Under condition \eqref{eq4.1}, with probability at least $1 - 2 \epsilon_1$, 
$$
\bigl\lvert \log(v) - \log(\wh{v}) \bigr\rvert \leq - \frac{1}{2} 
\log \biggl( 1 - \frac{\xi}{\delta} \biggr).
$$
\end{prop}
A simpler result is obtained by choosing $\ds \xi = 2y ( 1 + 2y)$, 
(the values of $y$ and $\delta$ being kept the same, so that 
we modify only the choice of $\beta_1$ through a different choice of $\xi$).
In this case, $Q(\beta_1) < -y$ as soon as 
$$
(1 + 2 y)^2 \leq \frac{ 2 + 2 \zeta \delta + \zeta \delta / y}{ 
1 + \zeta} = 2 + \frac{ 2 \zeta \delta + \zeta \delta / y - 2 \zeta}{1 + \zeta},
$$
which is true as soon as $\ds (1 + 2 y)^2 \leq 2$ and 
$\ds \frac{\delta}{y} \geq 2$, 
yielding the simplified condition
\begin{equation}
\label{eq4.2}
\log(\epsilon_1^{-1})  \leq \min \biggl\{ \frac{q}{4(1 + \sqrt{2})},
\frac{n-r}{8 \chi} \biggr\}.
\end{equation}
In this case, 
we get with probability at least $1 - 2 \epsilon_1$
that 
$$
\lvert \log(v) - \log(\wh{v}) \rvert \leq - \frac{1}{2}  
\log \biggl( 1 - \frac{2 y(1 + 2y)}{\delta} \biggr) 
\simeq \frac{y}{\delta}.
$$
\begin{prop}
\label{prop4.2}
Under condition \eqref{eq4.2}, 
with probability at least $1 - 2 \epsilon_1$, 
\begin{multline*}
\lvert \log(v) - \log(\wh{v}) \rvert \\ \leq - \frac{1}{2} 
\log \Biggl[ 1 - 2 \sqrt{\frac{ 2 \chi \log(\epsilon_1^{-1})}{n-r}} 
\biggl(1 + \frac{4 \log(\epsilon_1^{-1})}{q} \biggr) \Biggr]  
\\ \simeq \sqrt{ \frac{2 \chi \log(\epsilon_1^{-1})}{n}}.
\end{multline*}
\end{prop}
Recalling that $\ds \chi = \kappa - 1 + \frac{2}{p-1}$, we can choose 
in the previous proposition the 
approximately optimal block size 
$$
p = \biggl\lfloor \sqrt{ \frac{\ds n }{\ds (\kappa-1) \bigl[ 4 \log(\epsilon_1^{-1}) 
+ 1/2 \bigr]}} \biggr\rfloor.
$$
\begin{cor}
\label{cor4.3}
For this choice of parameter, as soon as the kurtosis (or its upper bound)
$\kappa \geq 3$, 
under the condition 
\begin{equation}
\label{eq4.4}
\log(\epsilon_1^{-1}) \leq  
\frac{n}{36(\kappa-1)} - \frac{1}{8},
\end{equation}
with probability at least $1 - 2 \epsilon_1^{-1}$, 
\begin{multline*}
\bigl\lvert \log(v) - \log(\wh{v}) \bigr\rvert \\ \leq - \frac{1}{2} 
\log \Biggl[ 1 - 2 \sqrt{ \frac{ 2 (\kappa-1) \log(\epsilon_1^{-1})}{n}} 
\exp \Biggl( 4 \sqrt{ \frac{4 \log(\epsilon_1^{-1}) + 1/2}{
(\kappa-1) n}} \; \Biggr) \Biggr] \\ 
\underset{n \rightarrow \infty}{\simeq} \sqrt{ \frac{2(\kappa-1) \log(
\epsilon_1^{-1})}{n}}.
\end{multline*}
\end{cor}
This is the asymptotics we hoped for, since the variance of $(Y_i - m)^2$ 
is equal to $(\kappa-1) v^2$. The proof is page \pageref{proof4.3}.

Let us plot some curves, showing the tighter bounds
of Proposition \thmref{prop4.1}, with optimal choice of $p$.
We compare our deviation bounds 
with the exact deviation quantiles 
of the variance estimate of equation \myeq{eq4.1bis} applied to 
a Gaussian sample, (given by a $\chi^2$ distribution). This 
demonstrates that we can stay of the same order under much weaker assumptions.

\noindent \mbox{} \hfill \raisebox{-1cm}[9cm][0cm]{\includegraphics{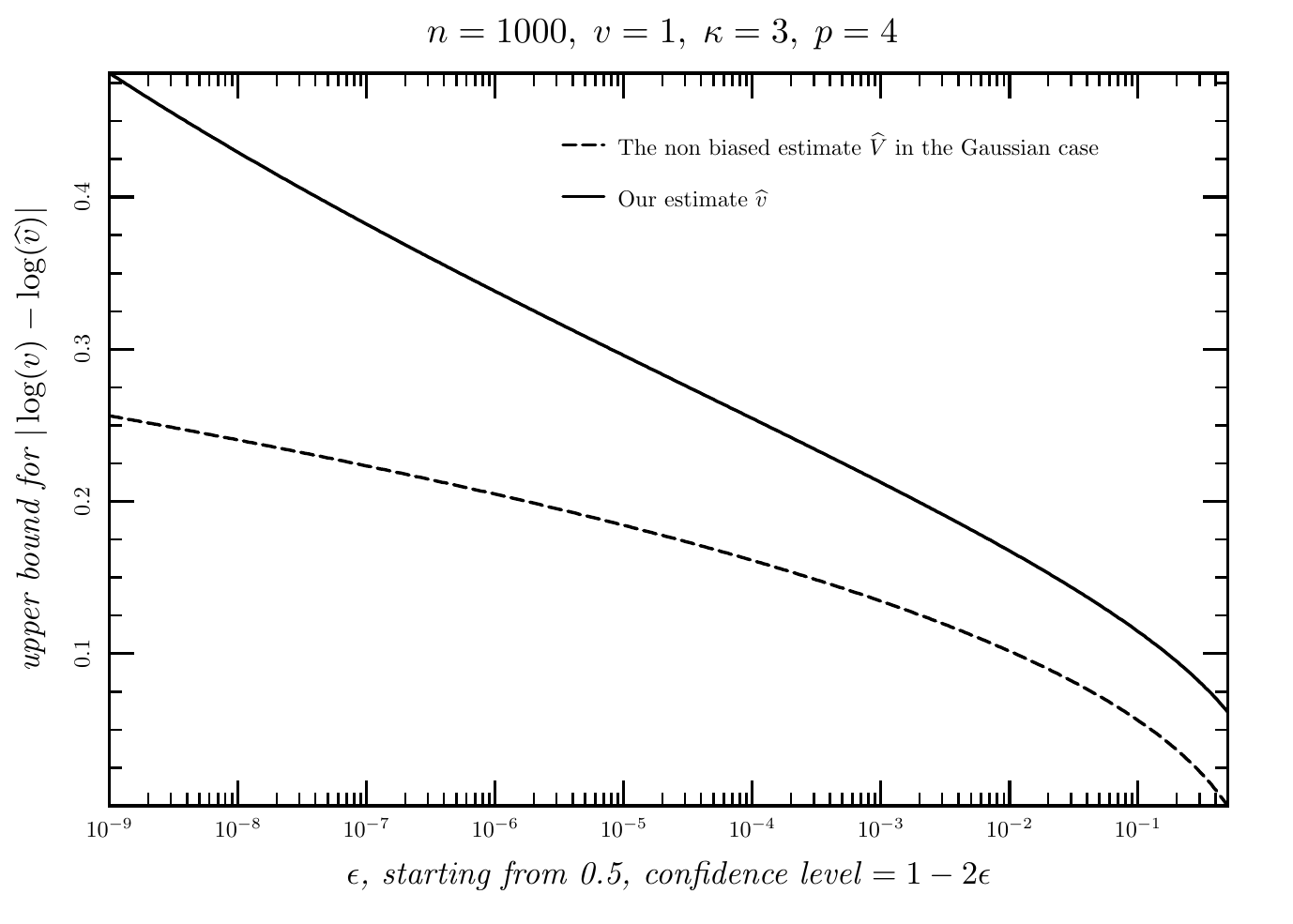}}
\hfill \mbox{}

\noindent \mbox{} \hfill \raisebox{-0.5cm}[9cm][0cm]{\includegraphics{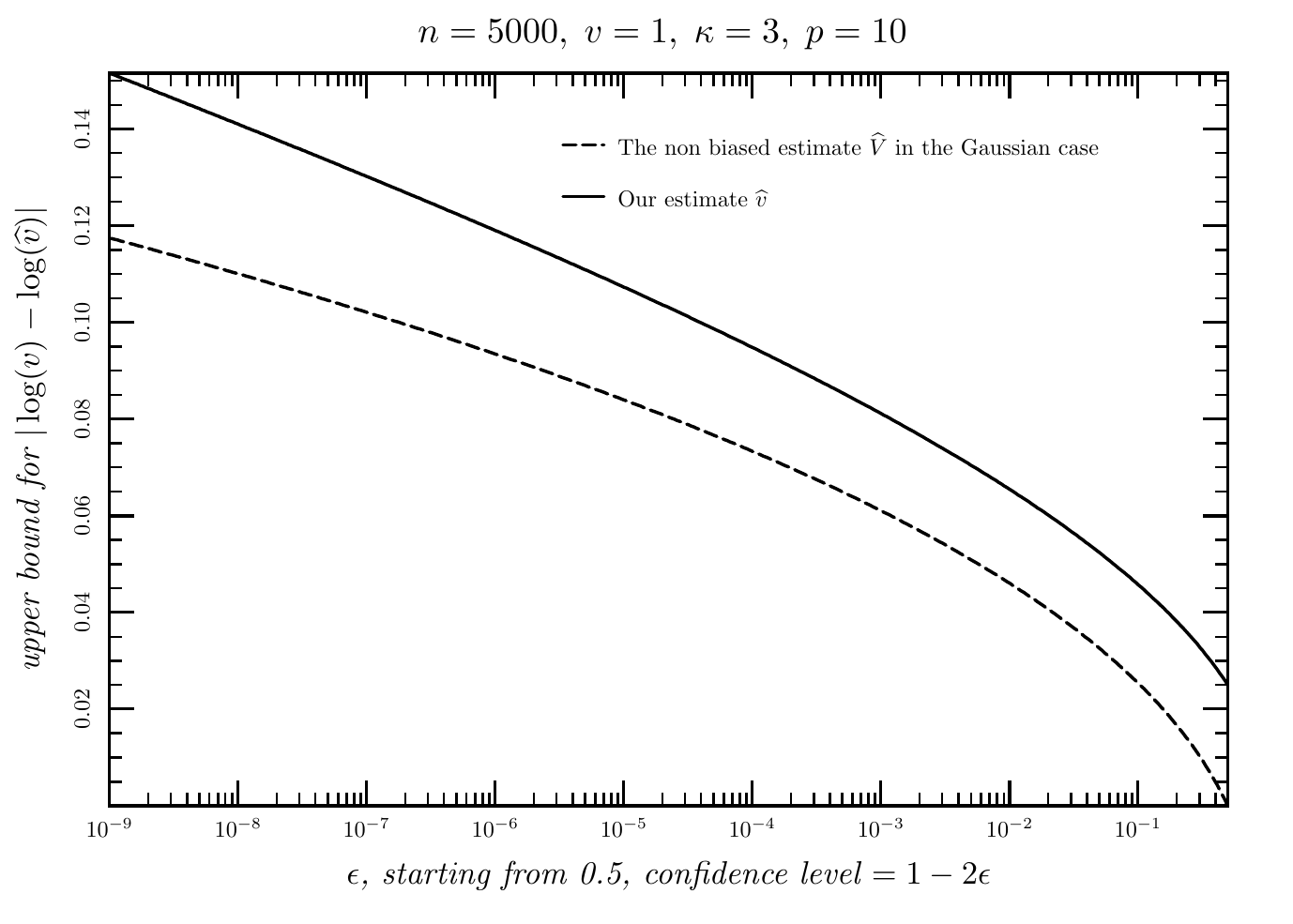}}
\hfill \mbox{}

\noindent \mbox{} \hfill \raisebox{-0.8cm}[9cm][0cm]{\includegraphics{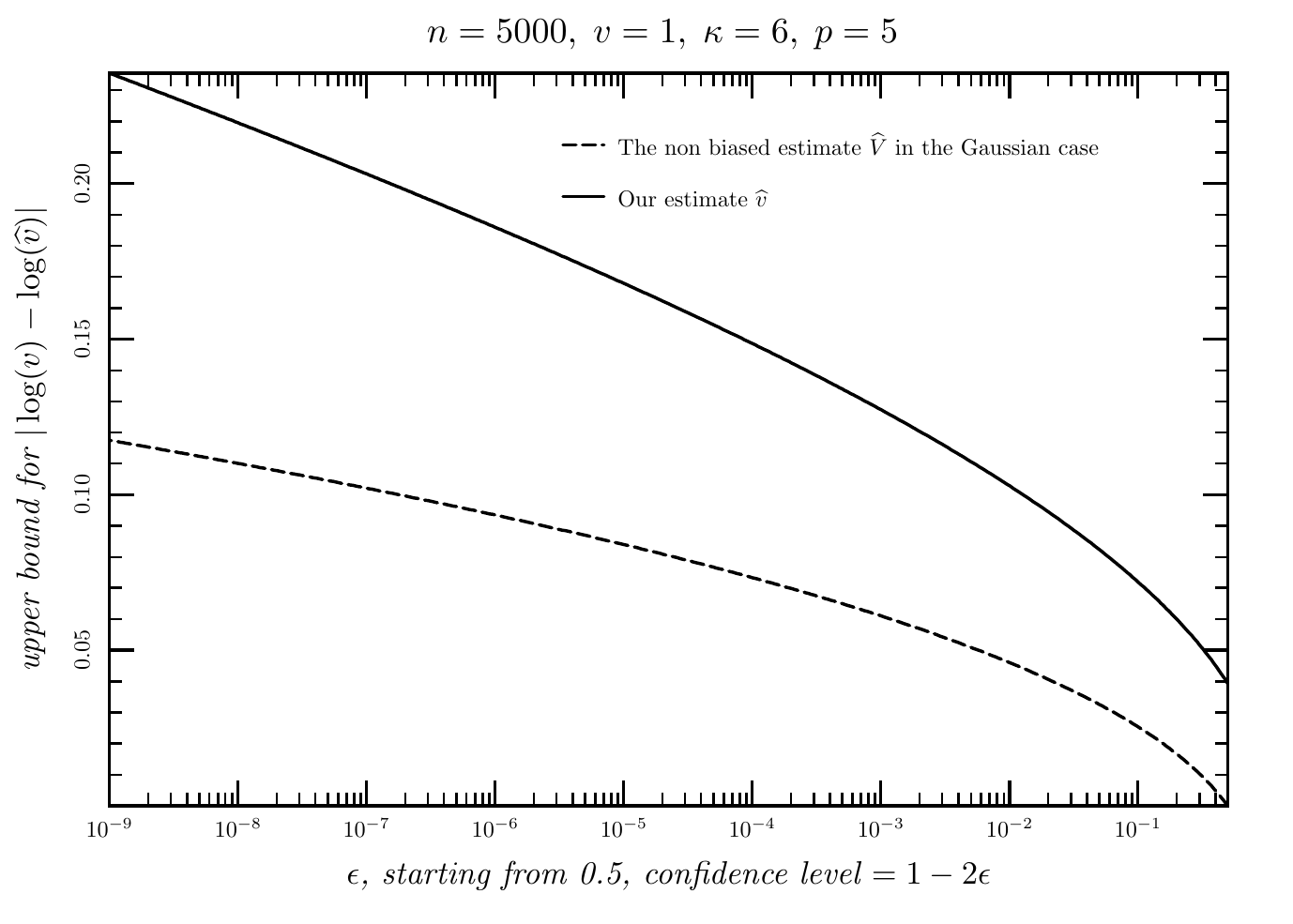}}
\hfill \mbox{}

\subsection{Mean estimate under a kurtosis assumption}

Here, we are going to plug a variance 
estimate $\wh{v}$ into a mean estimate. Let us therefore assume
that $\wh{v}$ is a variance estimate such that 
with probability at least $1 - 2 \epsilon_1$, 
$$
\bigl\lvert \log(v) - \log(\wh{v}) \bigr\rvert 
\leq \zeta.
$$ 
This estimate may for example be the one defined in the previous section.
Let $\wh{\alpha}$ be some estimate of the desired value of the 
parameter $\alpha$, to be defined later as a function of $\wh{v}$. 
Let us define $\wh{\theta} = \wh{\theta}_{\wh{\alpha}}$ by 
\begin{equation}
\label{eq4.11.2}
r(\wh{\theta}) = \frac{1}{n\wh{\alpha}}\sum_{i=1}^n \psi \bigl[ \wh{\alpha} (Y_i - \wh{\theta}) \bigr] 
= 0,
\end{equation}
where $\psi$ is the {\em narrow} influence function 
defined by equation \myeq{eq1.2}.
As usual, we are looking for non-random values 
$\theta_-$ and $\theta_+$ such that with large
probability $r(\theta_+) < 0 < r(\theta_-)$,
implying that $\theta_- < \wh{\theta} < \theta_+$.
But there is a new difficulty, caused by the fact 
that $\wh{\alpha}$ will be an estimate depending 
on the value of the sample. This problem will be 
solved with the help of PAC-Bayes inequalities.

To take advantage of PAC-Bayes theorems, we are 
going to compare $\wh{\alpha}$ with a perturbation 
$\wt{\alpha}$ built with the help of some supplementary
random variable.
Let indeed $U$ be some uniform real random variable on the  
interval $(-1,+1)$, independent from everything else. 
Let us consider 
$$
\wt{\alpha} = \wh{\alpha} + x \alpha \sinh \bigl( 
\zeta/2 \bigr) U.
$$
Let $\rho$ be the distribution of $\wt{\alpha}$ 
given the sample value. We are going to compare
this {\em posterior} distribution (meaning 
that it depends on the sample), with a {\em prior}
distribution that is independent from the sample. 
Let $\pi$ be the uniform probability distribution 
on the interval $\Bigl( \alpha \bigl[ \exp \bigl( - \zeta/2) - x 
\sinh \bigl( \zeta /2 \bigr)  
\bigr] ,
\alpha \bigl[ \exp \bigl( \zeta / 2 \bigr) + x \sinh 
\bigl( \zeta / 2 \bigr)  \bigr] \Bigr)$. 
Let us assume that $\wh{\alpha}$ and $\alpha$ are defined 
with the help of some positive constant $c$ as 
$\ds \wh{\alpha} = \sqrt{\frac{c}{\wh{v}}}$ and 
$\ds \alpha = \sqrt{\frac{c}{v}}$. In this case,
with probability at least $1 - 2 \epsilon_1$
\begin{equation}
\label{eq4.11}
\bigl\lvert \log (\alpha) - \log(\wh{\alpha}) \bigr\rvert 
\leq \frac{\zeta}{2}.
\end{equation}
As a result, with probability at least $1 - 2 \epsilon_1$,
$$
\C{K}( \rho, \pi) = \log \bigl( 1 + x^{-1} \bigr).
$$
Indeed, whenever equation \eqref{eq4.11} holds, 
the (random) support of $\rho$ is included in the 
(fixed) support of $\pi$, so that the relative entropy 
of $\rho$ with respect to $\pi$ is given 
by the logarithm of the ratio of the lengths of 
their supports.  

Let us now upper-bound $\psi \bigl[ \wh{\alpha}(Y_i - \theta) \bigr]$ 
as a suitable function of $\rho$. 
\begin{lemma}
\label{lem3.4}
For any posterior distribution $\rho$ and any $f \in L_2(\rho)$, 
$$
\psi \bigl[ \tint \rho(d \beta) f(\beta) \bigr] \leq \int \rho(d \beta) 
\log \Bigl[ 1 + f(\beta) + \frac{1}{2} f(\beta)^2 + \frac{a}{2} 
\Var_{\rho}(f) \Bigr],
$$
where $a \leq 4.43$ is the numerical constant of equation \myeq{eq3.4}.
\end{lemma}

Applying this inequality to $f(\beta) = \beta (Y_i - \theta)$ in 
the first place and $f(\beta) = \beta( \theta - Y_i)$ to get the 
reversed inequality, we obtain
\begin{align}
\psi \bigl[ \wh{\alpha}(Y_i - \theta) \bigr] 
& \leq \int \rho(d \beta) \log \Bigl\{ 1 + \beta( Y_i - \theta) \nonumber 
 \\ \label{eq3.5}& \qquad + \frac{1}{2} \Bigl[ \beta^2 + \frac{a}{3} x^2 \alpha^2 
\sinh(\zeta/2)^2 \Bigr] (Y_i - \theta)^2 \Bigr\}, \\
\psi \bigl[ \wh{\alpha}(\theta - Y_i) \bigr] 
& \leq \int \rho(d \beta) \log \Bigl\{ 1 + \beta(\theta - Y_i) \nonumber
 \\ \label{eq3.6} & \qquad + \frac{1}{2} \Bigl[ \beta^2 + \frac{a}{3} x^2 \alpha^2 
\sinh(\zeta/2)^2 \Bigr] (Y_i - \theta)^2 \Bigr\}.
\end{align}

Let us now recall the fundamental PAC-Bayes inequality, concerned 
with any function $(\beta, y) \mapsto f(\beta, y) \in \B{L}_1(\pi \otimes \B{P})$, where $\B{P}$ is 
the distribution of $Y$, such that $\inf f > -1$.   
\begin{multline*}
\B{E} \biggl\{ \exp \biggl( \int \rho(d \beta) \sum_{i=1}^n 
\log \bigl[ 1 + f(\beta, Y_i) \bigr] - n 
\log \bigl[ 1 + \B{E} \bigl( f(\beta, Y) \bigr) \bigr]  - 
\C{K}(\rho, \pi) \biggr) \biggr\}  
\\ \leq \B{E} \biggl\{ \int \pi(d \beta) \exp \biggl( 
\sum_{i=1}^n \log \bigl[ 1 + f(\beta, Y_i) \bigr] 
- n \log \bigl[ 1 +  \B{E} \bigl( f(\beta, Y) \bigr)  
\bigr] \biggr) \biggr\}  
= 1, 
\end{multline*}
according to \cite[page 159]{Cat01} and Fubini's lemma 
for the last equality.
Thus, from Chebyshev's inequality, with probability at least $1 - \epsilon_2$, 
\begin{multline*}
\int \rho( d \beta) \sum_{i=1}^n \log \bigl[ 1 + f(\beta, Y_i) \bigr] 
\\ < n \int \rho( d \beta) \log \bigl[ 1 + \B{E} \bigl( f(\beta, Y) \bigr) \bigr] + 
\C{K}(\rho, \pi) + \log \bigl(\epsilon_2^{-1} \bigr) 
\\ \leq n \int \rho(d \beta) \B{E} \bigl( f(\beta, Y) \bigr) + 
\C{K}(\rho, \pi) + \log \bigl(\epsilon_2^{-1}\bigr).
\end{multline*}

Applying this inequality to the case we are interested in, 
that is to equations \eqref{eq3.5} and \eqref{eq3.6}, 
we obtain with probability at least $1 - 2 \epsilon_1 - 2 \epsilon_2$ that 
\begin{multline*}
\wh{\alpha} r ( \theta_+) 
< \wh{\alpha} (m-\theta_+) +  
\frac{1}{2} \Bigl[ \wh{\alpha}^2 + \frac{(a+1)}{3} x^2 \alpha^2 \sinh(\zeta/2)^2 \Bigr] 
\bigl[ v + (m - \theta_+)^2 \bigr] \\ + \frac{\log \bigl(1 + x^{-1} 
\bigr) + \log(\epsilon_2^{-1})}{n}
\end{multline*}
and 
\begin{multline*}
\wh{\alpha} r ( \theta_-) 
> \wh{\alpha} (m - \theta_-) - \frac{1}{2} 
\Bigl[ \wh{\alpha}^2 + \frac{(a+1)}{3}x^2 \alpha^2 \sinh(\zeta/2)^2 \Bigr]  
\bigl[ v + (m - \theta_-)^2 \bigr] \\ - \frac{\log \bigl(1 + x^{-1} 
\bigr) + \log(\epsilon_2^{-1})}{n}.
\end{multline*}
Let us put $\theta_+ - m = m - \theta_- = \sqrt{\gamma v}$ and let us 
look for some value of $\gamma$ ensuring that $r(\theta_+) < 0 < r(\theta_-)$, 
implying that $\theta_- \leq \wh{\theta} \leq \theta_+$. 
Let us choose
\begin{align}
\nonumber
\alpha & = \sqrt{ \frac{ 2 \bigl[ \log(1 + x^{-1}) + \log(\epsilon_2^{-1}) \bigr]}
{ n \bigl[ 1 - \frac{(a+1)}{3} x^2 \sinh(\zeta/2)^2 \bigr] 
(1 + \gamma) v}}, \\
\label{eq4.15}
\wh{\alpha} & = 
\sqrt{ \frac{ 2 \bigl[ \log(1 + x^{-1}) + \log(\epsilon_2^{-1}) \bigr]}
{ n \bigl[ 1 - \frac{(a+1)}{3} x^2 \sinh(\zeta/2)^2 \bigr] 
(1 + \gamma) \wh{v}}} = \alpha \sqrt{\frac{v}{\wh{v}}},
\end{align}
assuming that $x$ will be chosen later on such that 
$$
\frac{(a+1)}{3} x^2 \sinh(\zeta/2)^2 < 1.
$$
Since 
$$
\frac{\log(1+x^{-1}) + \log(\epsilon_2^{-1})}{n} 
= \frac{\alpha^2}{2} \biggl( 1 - \frac{(a+1)}{3}x^2\sinh(\zeta/2)^2 \biggr) 
(1 + \gamma) v,
$$
we obtain with probability at least $1 - 2 \epsilon_1 - 2 \epsilon_2$ 
\begin{align*}
r(\theta_+) & < - \sqrt{\gamma v} + \frac{\alpha v (1 + \gamma)}{2} 
\biggl( \frac{\wh{\alpha}}{\alpha} + \frac{\alpha}{\wh{\alpha}} 
\biggr) \leq - \sqrt{\gamma v} + \alpha v (1 + \gamma) \cosh(\zeta/2), \\ 
\text{and } r(\theta-) & > \sqrt{\gamma v} - \frac{\alpha v (1 + \gamma)}{2} \biggl( 
\frac{\wh{\alpha}}{\alpha} + \frac{\alpha}{\wh{\alpha}} \biggr) 
\geq \sqrt{\gamma v} - \alpha v (1 + \gamma) \cosh(\zeta/2). 
\end{align*}
Therefore, if we choose $\gamma$ such that $\sqrt{\gamma v} = \alpha 
v(1 + \gamma) \cosh(\zeta/2)$, we obtain with probability at least $1 - 
2 \epsilon_1 - 2 \epsilon_2$ that $r(\theta_+) < 0 < r(\theta_-)$, 
and therefore that $\theta_- < \wh{\theta} < \theta_+$. 
The corresponding value of $\gamma$ is 
$\gamma = \eta / (1 - \eta)$, where 
$$
\eta = \frac{2 \cosh(\zeta/2)^2 \bigl[ \log(1 + x^{-1}) + \log(\epsilon_2^{-1}) \bigr]}{
n \bigl[ 1 - \frac{(a+1)}{3} x^2 \sinh(\zeta/2)^2 \bigr]}.
$$
\begin{prop}
\label{prop3.3}
With probability at least $1 - 2 \epsilon_1 - 2 \epsilon_2$, 
the estimator $\wh{\theta}_{\wh{\alpha}}$ defined by equation 
\myeq{eq4.11.2}, where $\wh{\alpha}$ is set as in 
\myeq{eq4.15}, satisfies
$$
\lvert \wh{\theta}_{\wh{\alpha}} - m \rvert \leq \sqrt{\frac{\eta v}{1 - \eta}} 
\leq \sqrt{\frac{\eta \wh{v}}{1-\eta}}\exp(\zeta/2) 
\leq \sqrt{\frac{\eta v}{1 - \eta}} \exp( \zeta).
$$
\end{prop}
The optimal value of $x$ is the one minimizing 
$$
\frac{\log(1 + x^{-1}) + \log(\epsilon_2^{-1}) }{\ds 1 - \frac{(a+1)}{3} x^2 \sinh(\zeta/2)^2}.
$$
Assuming $\zeta$ to be small, the optimal $x$ will be large, 
so that $\log(1 + x^{-1}) \simeq x^{-1}$, and we can choose the 
approximately optimal value
$$
x = \biggl[ \frac{2(a+1)}{3} \log(\epsilon_2^{-1})  
\biggr]^{-1/3} \sinh(\zeta/2)^{-2/3} .
$$
Let us discuss now the question of balancing $\epsilon_1$ 
and $\epsilon_2$. Let us put $\epsilon = \epsilon_1 + \epsilon_2$ 
and let $y = \epsilon_1/\epsilon$. Optimizing $y$ for a fixed 
value of $\epsilon$ could be done numerically, 
although it seems difficult to obtain a closed formula. 
However, the entropy term in $\eta$ can be written as
$\log(1 + x^{-1}) + \log(\epsilon^{-1}) - \log(1 - y)$. Since
$\zeta$ decreases, and therefore the almost optimal $x$ above 
increases,  when $y$ increases, we will get
an optimal order of magnitude (up to some constant less than $2$) for the bound if we balance $- \log(1 - y)$ 
and $\log(1 + x^{-1})$, resulting in the choice $y  = (1+x)^{-1}$, 
where $x$ is approximately optimized as stated above
(this choice of $x$ depends on $y$, so we end up 
with an equation for $x$, which can 
be solved using an iterative approach). 
This results, with probability at least $1 - 2 \epsilon$, 
in an upper bound for $\lvert \wh{\theta} - m \rvert$ 
 equivalent for large values of the sample size $n$ to
$$
\sqrt{ \frac{2 \log(\epsilon^{-1}) v}{n}}.
$$
Thus we recover, as desired, the same asymptotics as when 
the variance is known. 

Let us illustrate what we get when $n = 500$ or $n = 1000$ and it is known 
that $\kappa \leq 3$.

On these figures, we have plotted upper and lower bounds 
for the deviations of the empirical mean when the sample distribution\footnote{
$\C{B}_{1,\kappa}$ is defined by \myeq{eq1.2.0}} 
is the least favourable one in $\C{B}_{1,\kappa}$.
(These bounds will be proved later on. The upper bound is computed by taking the minimum of three bounds, 
explaining the discontinuities of its derivative). What we see 
on the $n = 500$ example is that our bound remains of the same order 
as the Gaussian bound up to confidence levels of order $1 - 10^{-8}$, 
whereas this is not the case with the empirical mean.  

In the case when $n = 1000$, we see that our estimator possibly improves on 
the empirical mean in the 
range of confidence levels
going from $1 - 10^{-2}$ to $1 - 10^{-6}$ and is a proved winner in the 
range going from $1 - 10^{-6}$ to  $1 - 10^{-14}$.

\noindent \mbox{} \hfill \raisebox{-1.75cm}[8cm][0cm]{\includegraphics{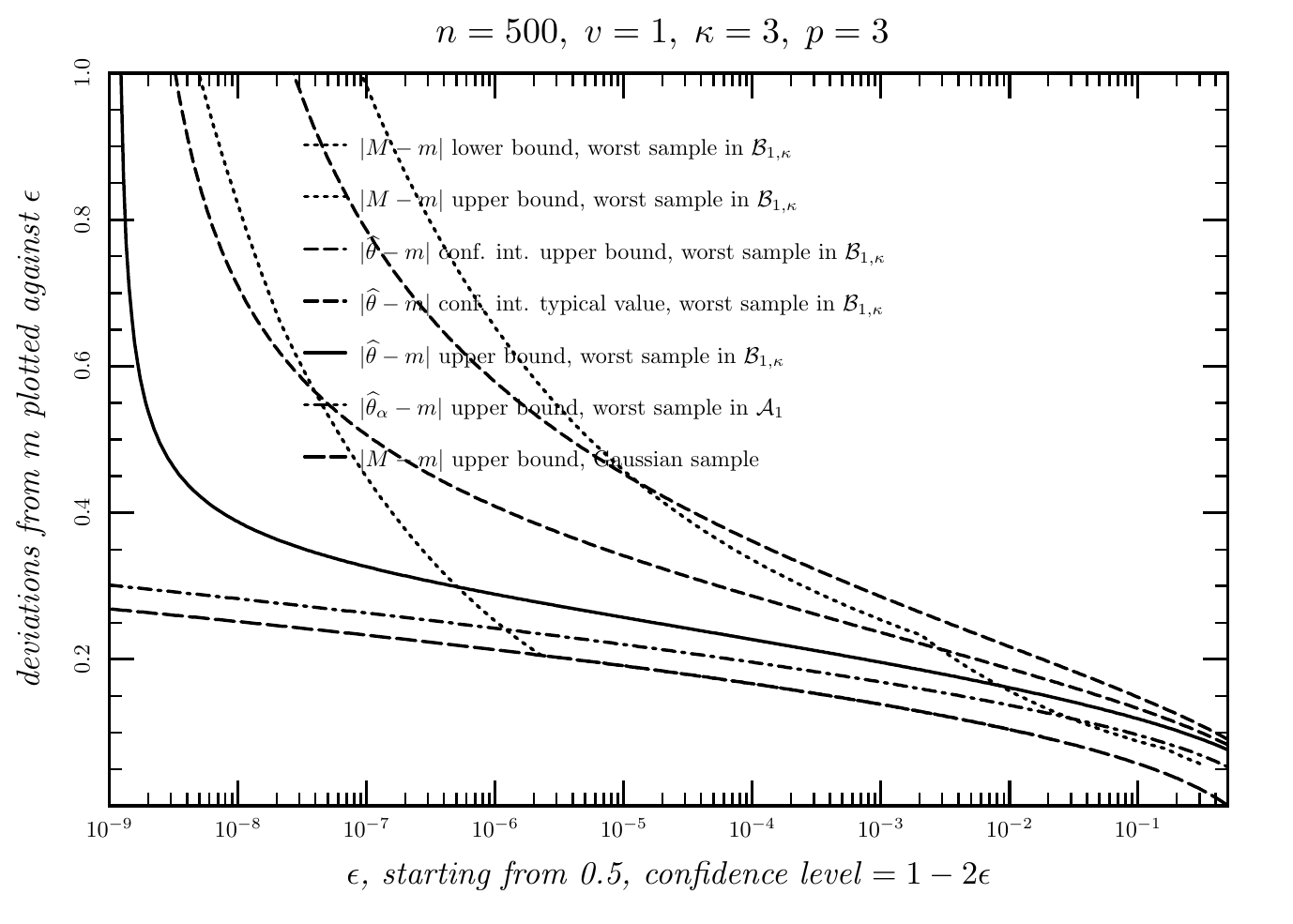}}
\hfill \mbox{}

\noindent \mbox{} \hfill \raisebox{-0.5cm}[9.5cm][0cm]{\includegraphics{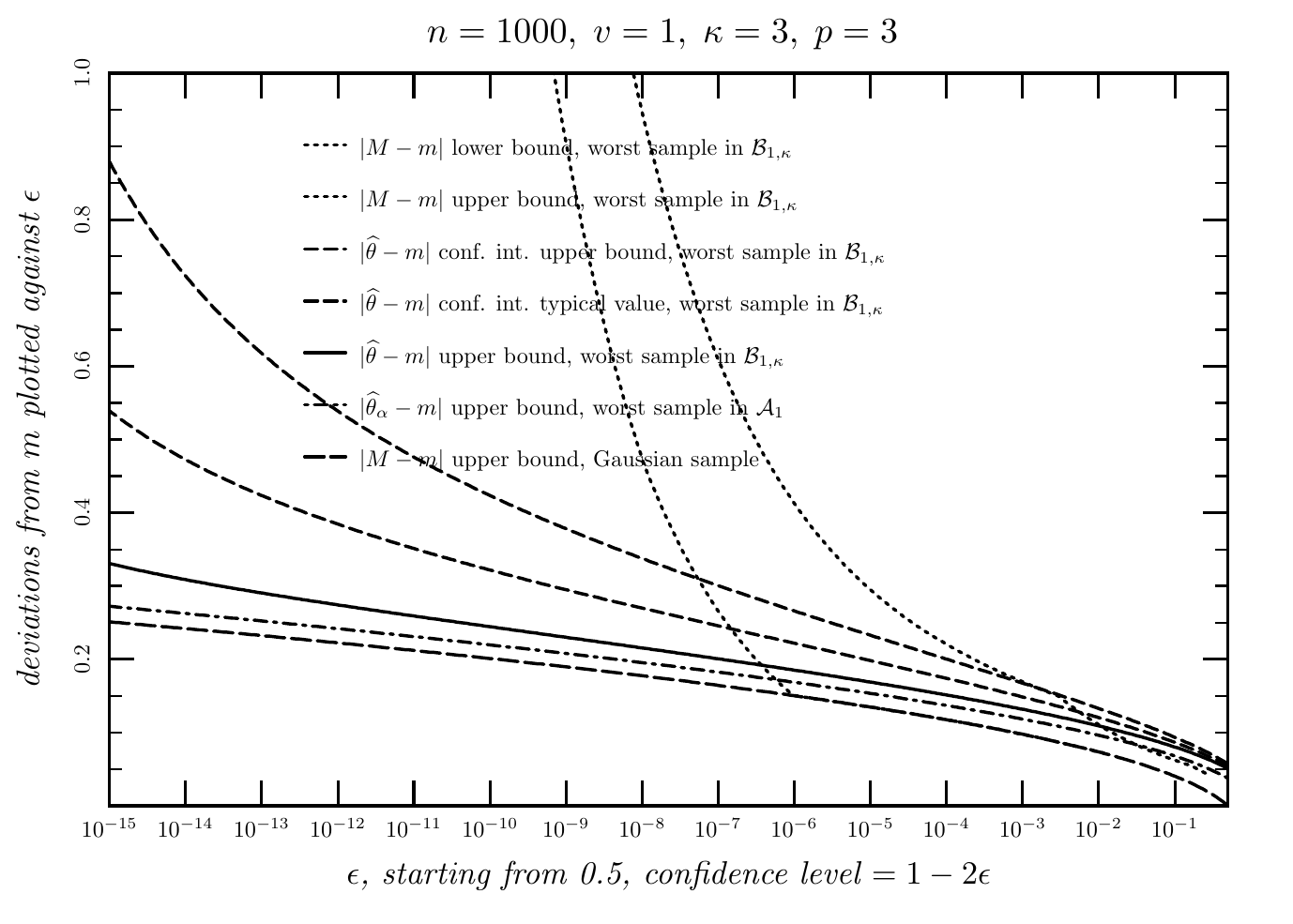}}   
\hfill \mbox{}

Let us see now the influence of $\kappa$ and plot the curves corresponding
to increasing values of $n$ and $\kappa$. 

\noindent \mbox{} \hfill \raisebox{0cm}[10cm][0cm]{\includegraphics{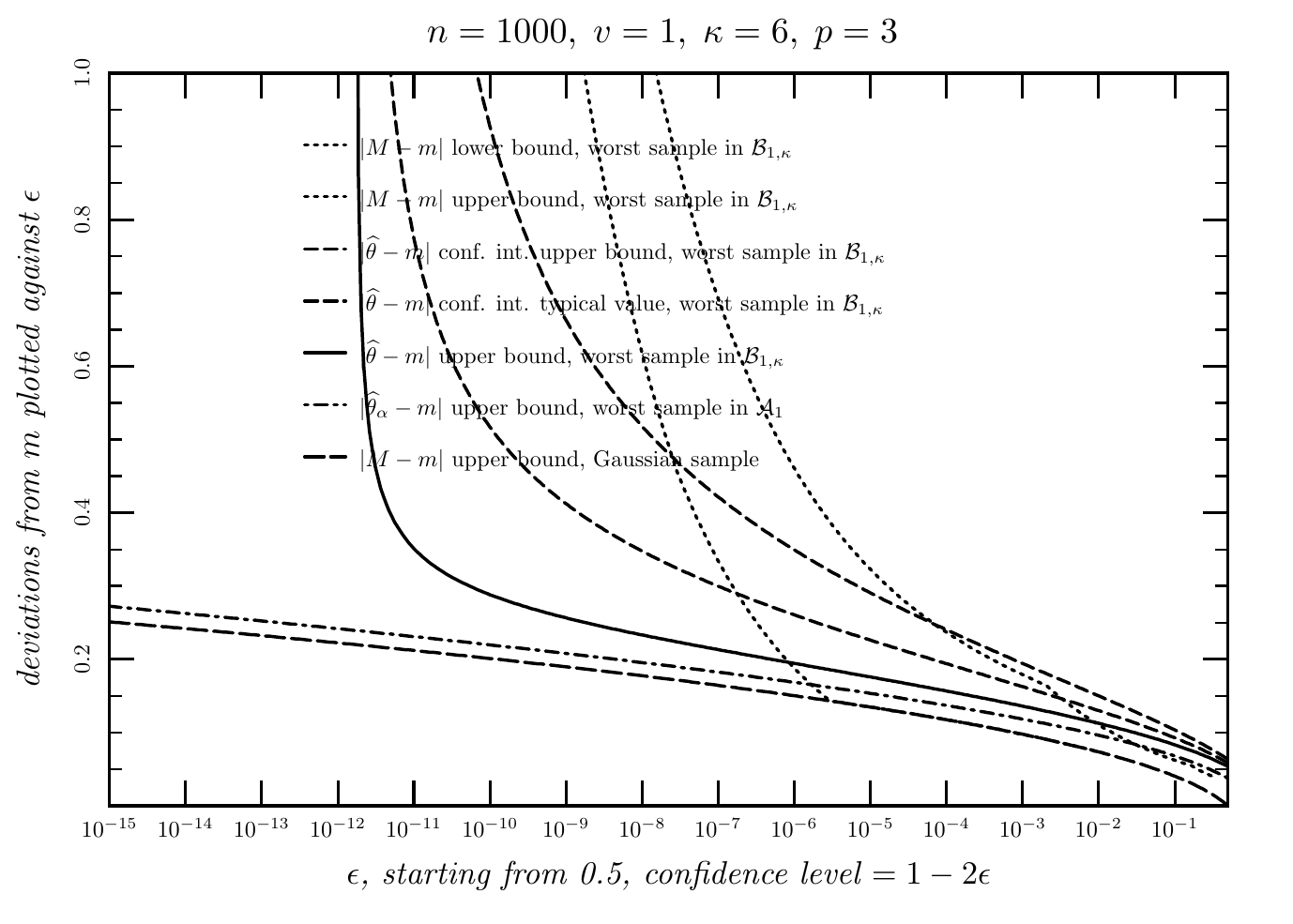}}   
\hfill \mbox{}

\noindent \mbox{} \hfill \raisebox{-0cm}[10cm][0cm]{\includegraphics{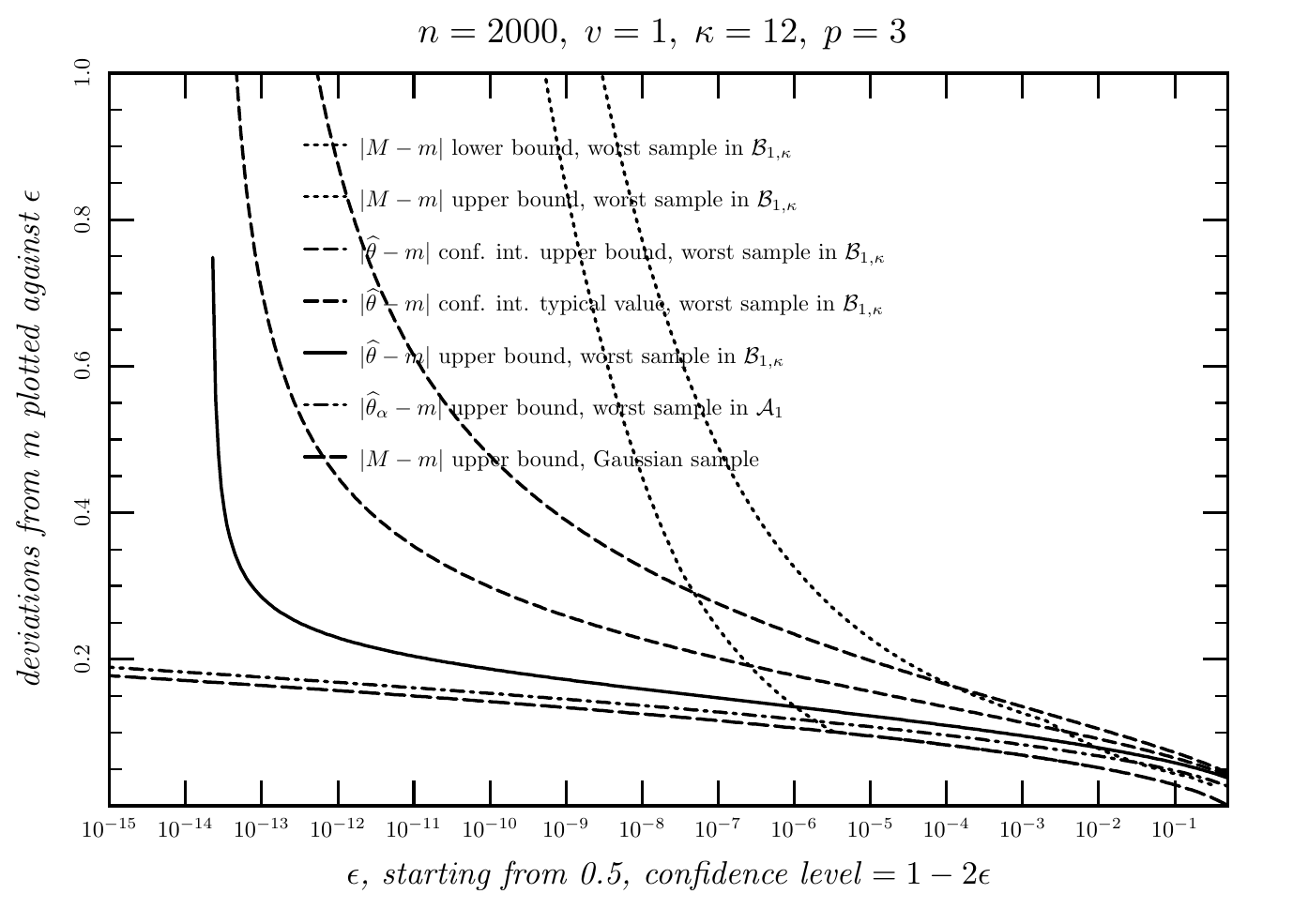}}   
\hfill \mbox{}\\
\mbox{} \hfill \raisebox{0cm}[10cm][0cm]{\includegraphics{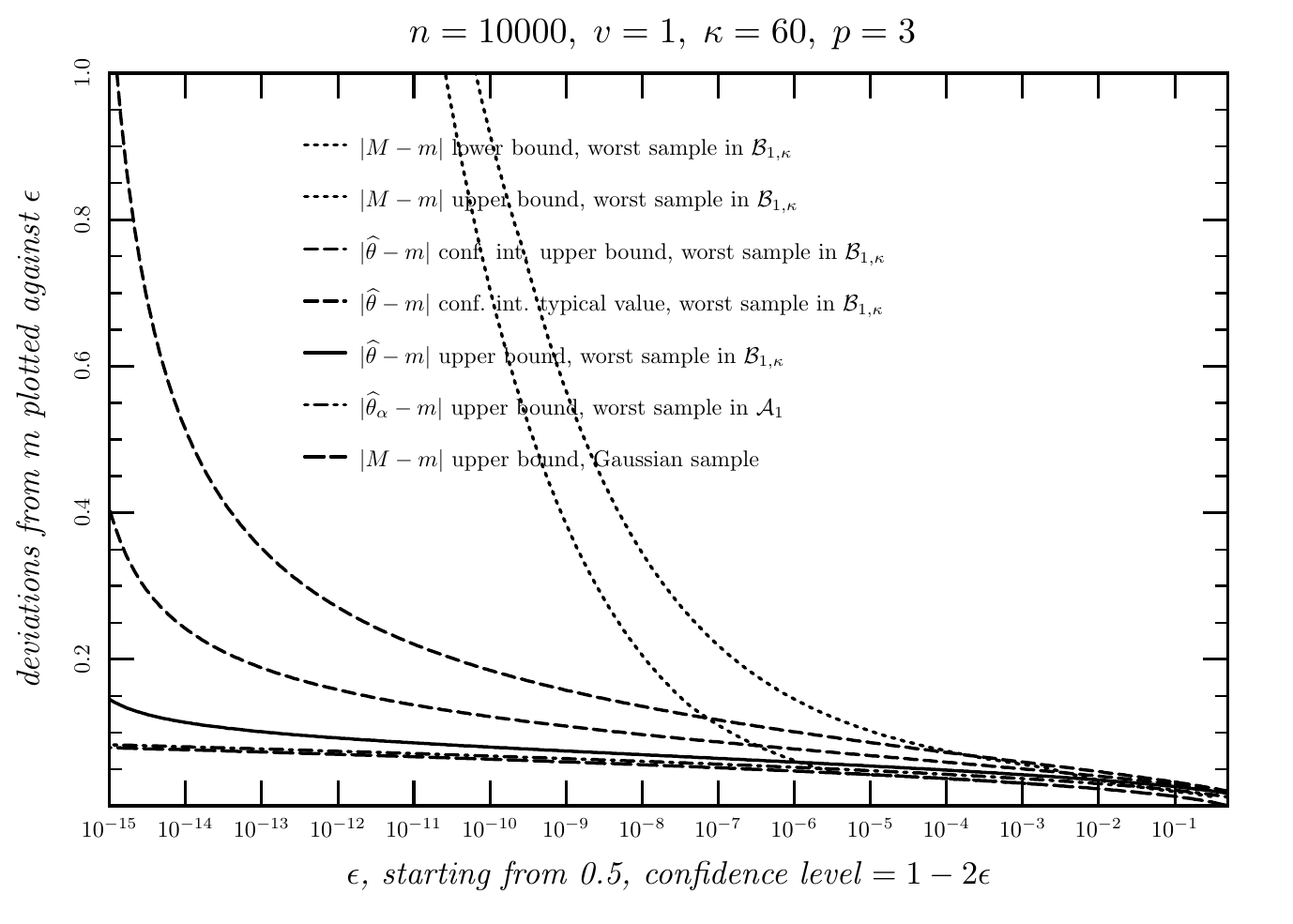}}   
\hfill \mbox{}\\
When we double the kurtosis letting $n = 1000$, we follow the Gaussian curve up to 
confidence levels of $1 - 10^{-10}$ instead of $1 - 10^{-14}$.
This is somehow the maximum kurtosis for which the bounds are satisfactory 
for this sample size. Looking at Proposition \thmref{prop4.2}, 
we see that the bound in first approximation depends on the ratio 
$\chi/n = \kappa/n$ (when $p = 3$), suggesting that to obtain similar performances, 
we have to take $n$ proportional to $\kappa$, which gives 
a minimum sample size of $n = 1000 \kappa/6$, if we want 
to follow the Gaussian curve up to confidence levels of order at least 
$1 - 10^{-10}$. 

These curves suggest another approach to choose the kurtosis parameter
$\kappa$. It is to use the largest value of the kurtosis with a low 
impact on the bound of Proposition \thmref{prop3.3}, 
given the sample size.  This leads, when in doubt about the true  
kurtosis value, for sample sizes $n \geq 1000$, 
to set, according to the previous 
rule of thumb, the kurtosis in the definition of the
estimator to the value $\ds \kappa_{\max} = 6 n/1000$. Doing so,
we get almost the same deviations as if the sample distribution 
were Gaussian, at levels of confidence up
to $1 - 10^{-10}$, for the largest range of (possibly non-Gaussian) sample
distributions. 

\section{Upper bounds for the deviations 
of the empirical mean}

In the previous sections, we compared new mean estimators with 
the empirical mean. We will devote the end of this paper
to prove the bounds on the empirical mean used in these 
comparisons.

This section deals with upper bounds, whereas the next one 
will study corresponding lower bounds.

Let us start with the case when 
the sample distribution may be any probability measure with a finite variance.
It is natural in this situation to bound the deviations 
of the empirical mean 
$$
M = \frac{1}{n} \sum_{i=1}^n Y_i
$$  
applying Chebyshev's inequality to its second moment,
to conclude that
\begin{equation}
\label{eq7.1}
\B{P} \biggl( \lvert M - m \rvert \geq \sqrt{\frac{v}{2 \epsilon n}} 
\biggr) \leq 2 \epsilon.
\end{equation}

This is in general close to optimal, as will be shown later
when we will compute corresponding lower bounds. 

When the sample distribution has a finite kurtosis, it is 
possible to take this into account to refine the bound. 
The analysis becomes less straightforward, and will be
carried out in this section.
The following bound uses a truncation 
argument, allowing to study separately the behaviour of small 
and large values. It is to our knowledge a new result. 
We will show later in this paper that its leading term is essentially 
tight --- up to a factor $\ds \biggl(\frac{\kappa}{\kappa-1}\biggr)^{1/4}$,
--- 
when the proper asymptotic is considered. 

\begin{prop}
\label{prop6.1}
For any probability distribution whose kurtosis is not greater than
$\kappa$, the empirical mean $M$ is such that with probability at 
least $1 - 2 \epsilon$, 
\begin{multline*}
\frac{\lvert M - m \rvert}{\sqrt{v}} \leq \inf_{\lambda \in (0,1)} \sqrt{ \frac{2 \log(\lambda^{-1} \epsilon^{-1})}{n}} 
+ \frac{\sqrt{\kappa} \log(\lambda^{-1} \epsilon^{-1})}{3n} 
\\ + \left( \frac{\kappa}{2(1-\lambda) n^3 \epsilon} 
\right)^{1/4} \biggl[ 1 + \frac{3(n-1) \kappa 
\log(\lambda^{-1} \epsilon^{-1})^2}{
4^3(1 + \sqrt{2})^4 n^2} \biggr]^{1/4} 
\underset{\substack{n \epsilon \rightarrow 0 \\ \epsilon^{1/n} \rightarrow 1 }}{\simeq} 
\left( \frac{\kappa}{2 n^3 \epsilon} \right)^{1/4}.
\end{multline*}
Instead of minimizing the bound in $\lambda$, 
one can also take for simplicity 
$$
\lambda = \min \biggl\{ \frac{1}{2}, 
2^{7/4} \biggl( \frac{n \epsilon}{\kappa} \biggr)^{1/4} \sqrt{\log \biggl( \frac{\kappa}{2n\epsilon^5} 
\biggr)} 
\biggr\}.
$$
\end{prop}
We see that there are two regimes in the behaviour of the deviations of $M$. 
A Gaussian regime for levels of confidence less than 
$1 - 1/n$ and long tail regime for higher confidence levels, 
depending on the value of the kurtosis $\kappa$.

In addition to this, let us also put forward the fact that, even in the simple case when 
the mean $m$ is known, estimating the 
variance under
a kurtosis hypothesis at high confidence levels 
cannot be done using  
the empirical estimator  
$$
M_2 = \frac{1}{n} \sum_{i=1}^n (Y_i - m)^2.
$$
Indeed, assuming without loss of generality 
that $m = 0$ and computing the quadratic mean 
$$
\B{E} \Bigl\{ \bigl[ M_2 - \B{E}(Y^2) \bigr]^2 \Bigr\} 
= \frac{\B{E}(Y^4) - \B{E}(Y^2)^2}{n} = \frac{(\kappa-1)}{n} \B{E}(Y^2)^2,
$$
we can only conclude, using Chebyshev's inequality,  
that with probability at least $1 - 2 \epsilon$
$$
\B{E}(Y^2) \leq \frac{M_2}{\ds 1 - \sqrt{\frac{\kappa-1}{2n\epsilon}}},
$$
a bound which blows up at level of confidence 
$\ds \epsilon = \frac{\kappa-1}{2n}$, and which we do not suspect to 
be substantially improvable in the worst case. 
In contrast to this, Propositions \ref{prop4.1} 
and \thmref{prop4.2}  
provide variance estimators with high confidence levels. 

\section{Lower bounds}

\subsection{Lower bound for Gaussian sample distributions}

This lower bound is well known. We recall it here for the sake
of completeness. 

The empirical mean has optimal deviations when the sample distribution 
is Gaussian in the following precise sense.

\begin{prop}
\label{prop2.1.2}
For any estimator of the mean $\wh{\theta} : \B{R}^n \rightarrow \B{R}$, 
any variance value $v > 0$, 
and any deviation level $\eta > 0$, 
there is some Gaussian measure $\C{N}(m,v)$ (with variance $v$ and mean $m$)
such that the i.i.d. sample of length $n$ drawn from this distribution 
is such that 
$$
\B{P} \bigl( \wh{\theta} \geq m + \eta \bigr) \geq \B{P} 
\bigl( M \geq m + \eta \bigr) 
\quad \text{ or } \quad 
\B{P} \bigl( \wh{\theta} \leq m - \eta \bigr) \geq 
\B{P} \bigl( M \leq m - \eta \bigr),
$$
where $\ds M = \frac{1}{n} \sum_{i=1}^n Y_i$ is the empirical mean.
\end{prop}

This means that any distribution free symmetric confidence interval based on 
the (supposedly known) value of the variance has to include the 
confidence interval for the empirical mean of a Gaussian distribution, 
whose length is exactly known and equal to the properly
scaled quantile of the Gaussian measure. 

Let us state this more precisely. With the 
notations of the previous proposition 
\begin{multline*}
\B{P}(M \geq m + \eta) = \B{P} \bigl( M \leq m - \eta) 
\\ = G\left[ \left( \sqrt{\frac{n}{v}} \eta , + \infty \right( 
\right] = 1 - F \left( \sqrt{\frac{n}{v}} \eta \right),
\end{multline*}
where $G$ is the standard normal measure and $F$ its distribution 
function. 

The upper bounds proved in this paper can be decomposed into 
$$
\B{P}(\wh{\theta} \geq m + \eta) \leq \epsilon \qquad \text{and} 
\qquad \B{P} (\wh{\theta} \leq m - \eta) \leq \epsilon,
$$
although we preferred for simplicity to state them 
in the slightly weaker form $\B{P}( \lvert 
\theta - m \rvert \geq \eta) \leq 2 \epsilon$.

As the Gaussian shift model made of Gaussian sample distributions with 
a given variance and varying means, is included in all the 
models we are considering in this paper, we necessarily should have 
according to the previous proposition
$$
\epsilon \geq 1 - F \left( \sqrt{\frac{n}{v}} \eta \right), 
$$
which can be also written as 
$$
\eta \geq \sqrt{\frac{v}{n}} F^{-1} ( 1 - \epsilon ).
$$

Therefore some visualisation of the quality of our bounds can be
obtained by plotting $\ds \epsilon \mapsto \eta$ against 
$\ds \epsilon \mapsto \sqrt{\frac{v}{n}} F^{-1}(1 - \epsilon)$, 
as we did in the previous sections. 

Let us remark eventually that the assumed symmetry of the confidence 
region is not a real limitation. Indeed, if we can prove 
for any given estimator $\wh{\theta}$ that for 
any Gaussian sample distribution with a given variance $v$, 
\begin{align*}
\B{P} \Bigl[ m \geq \wh{\theta} + \eta_+(\epsilon) 
\Bigr] & \leq \epsilon, \\ 
\text{and } \quad \B{P} \Bigl[ m \leq \wh{\theta} - \eta_-(\epsilon) 
\Bigr] & \leq \epsilon,
\end{align*}
then we may consider for any value of $\epsilon$ the estimator
with symmetric confidence levels defined as 
$$
\wh{\theta}_s = \wh{\theta} + \frac{\eta_+(\epsilon)
- \eta_-(\epsilon)}{2}. 
$$
This symmetric estimator is such that for any Gaussian sample distribution with variance
$v$, 
\begin{align*}
\B{P} \Bigl[ m \geq \wh{\theta}_s + \frac{\eta_-(
\epsilon) + \eta_+(\epsilon) }{2}
\Bigr] & \leq \epsilon, \\ 
\B{P} \Bigl[ m \leq \wh{\theta}_s - 
\frac{\eta_-(\epsilon) +\eta_+(\epsilon)}{2}
\Bigr] & \leq \epsilon.
\end{align*}
Thus, applying the previous proposition, we obtain that
$$
\frac{\eta_+(\epsilon) + \eta_-(\epsilon)}{2} 
\geq \sqrt{\frac{v}{n}} F^{-1}(1 - \epsilon).
$$

\subsection[Deviations of the empirical mean depending on the 
variance]{Lower bound for the deviations of the empirical mean
depending on the variance} 

In the following proposition, we state a lower bound 
for the deviations of the empirical mean when the sample
distribution\footnote{$\C{A}_{v_{\max}}$ is defined by \myeq{eq1.1}} is the least favourable in $\C{A}_{v_{\max}}$ 
(meaning the distribution for which the deviations 
of the empirical mean are the largest). 

\begin{prop}
\label{prop2.2}
For any value of the variance $v$, any deviation level $\eta > 0$,
there is some distribution with variance $v$ and mean $0$ such that
the i.i.d. sample of size $n$ drawn from it satisfies  
$$
\B{P}(M \geq \eta) = \B{P}(M \leq - \eta) \geq 
\frac{v \left( 1 - \frac{v}{\eta^2 n^2} \right)^{n-1}}{2 n \eta^2}.
$$
Thus, as soon as $\epsilon \leq (2e)^{-1}$, 
with probability at least $2 \epsilon$, 
$$
\lvert M - m \rvert \geq  \sqrt{\frac{v}{2 n \epsilon}} 
\left( 1 - \frac{2e \epsilon}{n} \right)^{\frac{n-1}{2}}.
$$
\end{prop}
Let us remark that this bound is pretty tight, 
since, according to equation \myeq{eq7.1}
with probability at least $1 - 2 \epsilon$, 
$$
\lvert M - m \rvert \leq \sqrt{\frac{v}{2n\epsilon}}.
$$
This can also be observed on the plots following
Proposition \thmref{prop1.4}. 

\subsection[Deviations of the empirical mean depending 
on the kurtosis]{Lower bound for the deviations of empirical mean 
depending on the variance and the kurtosis}
Let us now refine the previous lower bound by taking into 
account the kurtosis $\kappa$ of the sample distribution, assuming 
of course that it is finite. 
\begin{prop}
\label{prop7.3}
As soon as $\epsilon^{-1} \geq n \geq 16$, there exists a sample 
distribution with mean $m$, finite variance $v$ and 
finite kurtosis $\kappa$, such that with probability at least $2 \epsilon$,
\begin{multline*}
\lvert M - m \rvert \geq \max \Biggl\{ \biggl[ \frac{(\kappa-1)(1 - 8 \epsilon)}{4n\epsilon} \biggr]^{1/4}, \\
\Biggl[ \frac{(\kappa-1)}{2 n \epsilon} \biggl[ 1 - \left( \frac{ n \epsilon}{16} \right)^{1/4} 
- 4 \epsilon \biggr] \Biggr]^{1/4} - 
\sqrt{\frac{\log\bigl[ 16 / (n \epsilon) \bigr] v}{2n}} \Biggr\} 
\sqrt{\frac{v}{n}}. 
\end{multline*}
\end{prop}
Let us remark that the asymptotic behaviour of this lower bound when 
$n \epsilon$ and $\log(\epsilon^{-1})/n$ both tend to zero matches 
the upper bound of Proposition \thmref{prop6.1} up to a multiplicative 
factor $\left( \frac{\kappa}{\kappa-1} \right)^{1/4} \leq 1.11$ 
when the kurtosis is $\kappa \geq 3$, which is the kurtosis of the 
Gaussian distribution.

The plots following Proposition \thmref{prop3.3} show that this lower
bound is not too far from the upper bound obtained by combining 
Proposition \thmref{prop6.1}, Proposition \thmref{prop7.3.2} and 
equation \myeq{eq7.1}. 

\section{Proofs}

\subsection{Proof of Corollary \thmref{cor4.3}} 
\label{proof4.3}
Let us remark first that condition \myeq{eq4.4} implies condition 
\myeq{eq4.2}, as can be easily checked. 
Putting
$\ds x = \sqrt{\frac{n}{(\kappa-1) \bigl[ 4 \log(\epsilon_1^{-1}) + 1/2 
\bigr]}}$, so that $p = \lfloor x \rfloor$,  
we can also check that $p-1 \geq x/2$ and $n - p + 1 \geq n/2$.
We can then write 
\begin{multline*}
\sqrt{\frac{2 \chi \log(\epsilon_1^{-1})}{n-r)}} \biggl( 1 + 
\frac{4 \log(\epsilon_1^{-1})}{q} \biggr) \\ \leq 
\sqrt{\frac{2 (\kappa-1) \log(\epsilon_1^{-1})}{n}} 
\exp \biggl( \frac{1}{(\kappa-1)(p-1)} \\ \shoveright{+ \frac{p-1}{2(n - p +1)} 
+ \frac{4 \log(\epsilon_1^{-1}) p}{n-p+1} \biggr) }\\ 
\leq 
\sqrt{\frac{2 (\kappa-1) \log(\epsilon_1^{-1})}{n}} 
\exp \biggl( \frac{2}{(\kappa-1)x} + \frac{2 x \bigl[ 4 \log(\epsilon_1^{-1}) 
+ 1/2 \bigr]}{n} \biggr) 
\\ 
= \sqrt{\frac{2(\kappa-1) \log(\epsilon_1^{-1})}{n}} \exp \Biggl( 
4 \sqrt{\frac{4 \log(\epsilon_1^{-1}) + 1/2}{(\kappa-1)n}} \Biggr). 
\end{multline*}

\subsection{Proof of Lemma \thmref{lem3.4}}

Let us introduce some modification of $\psi$ in order to improve the compromise
between $\inf \psi''$ and $\sup \psi$.  
Let us put $\ov{\psi}(x) = \log (1 + x + x^2/2)$. We would like to 
squeeze some function $\chi$ between $\psi$ and $\ov{\psi}$, in such a way that 
$\inf \chi'' = \inf \ov{\psi}''$. This will be better than 
using $\psi$ itself since 
\begin{align*}
\inf \ov{\psi}'' & = - 1/4 \\ 
\text{whereas } \inf \psi'' & = - 2. 
\end{align*}
Indeed these two values can be computed in the following way. 
Let us put $\varphi(x) = \exp \bigl[\, \ov{\psi}(x) \bigr] = 
1 + x + x^2/2$. It is easy to check that 
\begin{align*}
\ov{\psi}''(x) & = - \varphi(x)^{-1} \bigl[ 1 - \varphi(x)^{-1} \bigr], \\
\psi''(x) & = \varphi(-x)^{-1} \big[ 1 - \varphi(-x)^{-1} \bigr],
\end{align*}
implying that $ \ov{\psi}''(x) \geq - 1/4$. This inequality becomes
an equality 
when $\varphi(x) = 2$, that is when $x = \sqrt{3} - 1 \simeq 0.73$. 
In the same way $\psi''(x) \geq -2$ and equality is reached 
when $\varphi( - x) = 1/2$, that is when $x = 1$. We are going to build a function 
$\chi$ which follows $\psi$ when $x \leq x_1$, where $x_1$ satisfies 
$\psi''(x_1) = -1/4$. The value of $x_1$ is computed from the equation 
$\varphi(-x_1)^{-1} = (1 + \sqrt{2})/2$. Let $y_1 = \psi(x_1)$ and 
$p_1 = \psi'(x_1)$. They have the following values
\begin{align*}
x_1 & = 1 - \sqrt{4 \sqrt{2} - 5} \simeq 0.1895\\
y_1 & = - \log \bigl[ 2 \bigl( \sqrt{2} - 1 \bigr) \bigr] 
\simeq 0.1882, \\
p_1 & = \frac{ \sqrt{4 \sqrt{2} - 5}}{2 (\sqrt{2} - 1)} \simeq 0.978.
\end{align*}
After $x_1$, we continue $\chi$ with a quadratic function, until 
its derivatives cancels. 
Thus, the second derivative of $\chi$ being less than the second
derivative of $\ov{\psi}$ at each point of the positive real 
line, we are sure that $\chi(x) \leq \ov{\psi}(x)$ for 
any $x \in \B{R}$. The function $\chi$ built in this way 
satisfies the equation 
$$
\chi(x) = 
\begin{cases}
\psi(x), & x \leq x_1, \\
\ds y_1 + p_1(x-x_1) - \frac{1}{8}(x-x_1)^2, & x_1 \leq x \leq x_1 + 4 p_1, \\
y_1 + 2 p_1^2 \leq 2.103, &  x \geq x_1 + 4 p_1. 
\end{cases}
$$
As we have proved, and as can be seen on the next plot, 
the function $\chi$ is such that 
$$
\psi(x) \leq \chi(x) \leq \ov{\psi}(x), \qquad x \in \B{R}.
$$
\mbox{} \hfill \raisebox{-0.8cm}[9.2cm][0cm]{\includegraphics{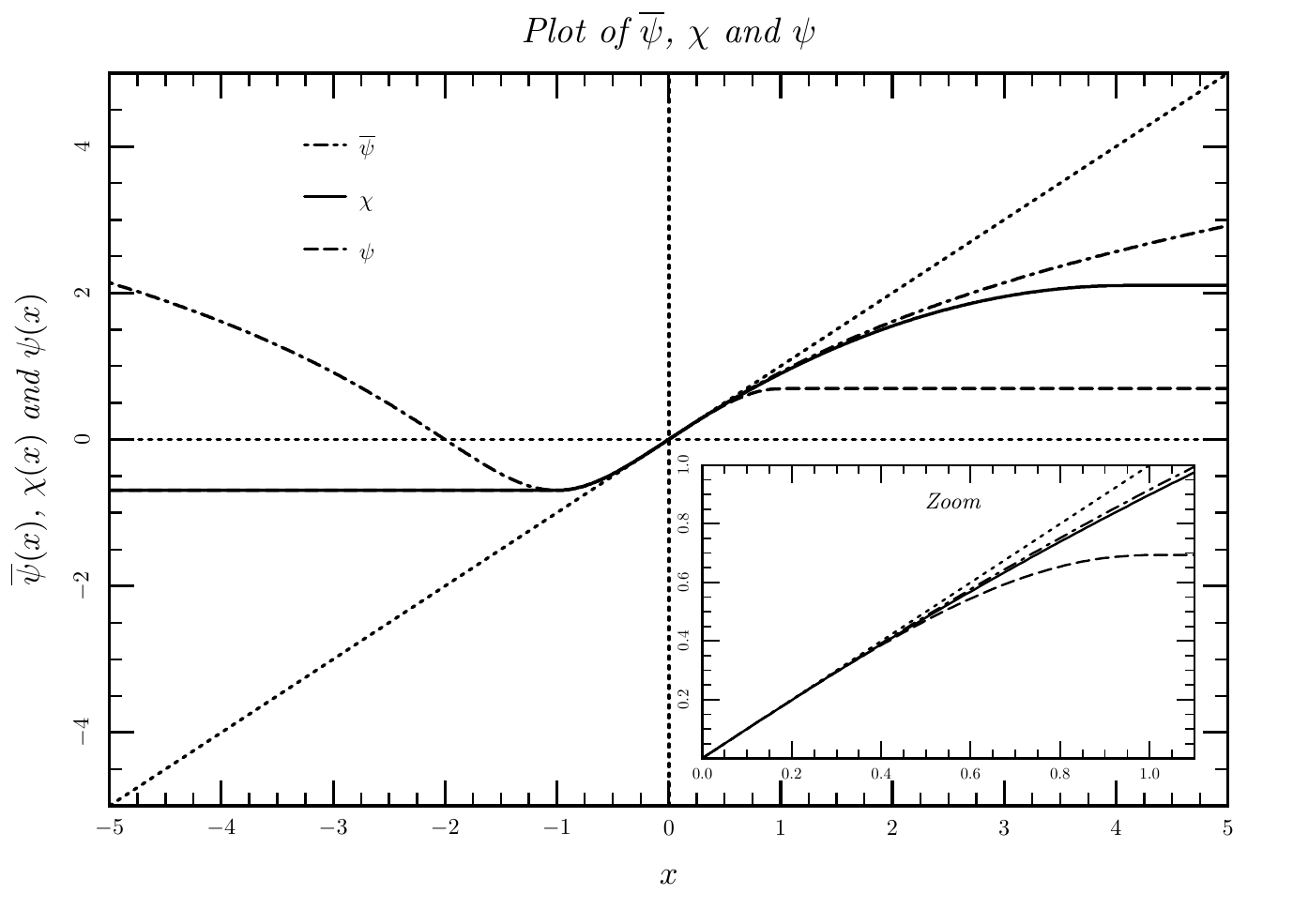}}
\hfill \mbox{}\\
Let us now compare $\chi$ with a suitable convex function (in order 
to apply Jensen's inequality).
Let us introduce to this purpose the function 
$\ov{\chi}_{x_*} = \chi(x) + \frac{1}{8}(x - x_*)^2$, 
which 
is convex for any choice of the parameter $x_* \in \B{R}$.

Let us consider as in the lemma we have to prove some 
function $f \in L_2(\rho)$ and let us choose
$x_* = \int \rho(d \beta) f(\beta)$
and 
put $\tint \rho(d \beta) \bigl[ f(\beta) - \tint \rho f 
\bigr]^2 = \Var_{\rho}(f)$. 
We obtain by Jensen's inequality 
\begin{multline*}
\psi (x_*)  
\leq \chi (x_*) = \ov{\chi}_{x_*} ( x_*)     
= \ov{\chi}_{x_*} \bigl[ \tint \rho( d \beta) f( \beta ) \bigr]    
\leq \tint \rho( d \beta) \ov{\chi}_{x_*}  \bigl[ f( \beta) 
\bigr] \\ = \tint \rho( d \beta) \chi \bigl[ f(\beta) \bigr]
+ \frac{1}{8} \Var_{\rho}(f). 
\end{multline*}
On the other hand, 
it is clear that
$$
\psi (x_*) \leq \tint \rho(d\beta) \chi \bigl[ 
f(\beta) \bigr] - \inf \chi  + \sup \psi  
= \tint \rho(d \beta) \chi \bigl[ f(\beta) \bigr] + \log(4).
$$
We have proved 
\begin{lemma} For any posterior distribution $\rho$ and any $f \in L_2(\rho)$, 
$$
\psi \bigl[ \tint \rho( d \beta) f(\beta) \bigr] 
\leq \tint \rho(d \beta) \chi \bigl[ f(\beta) \bigr] 
+ \min \Bigl\{ \log(4), \frac{1}{8}\Var_{\rho}(f) \Bigr\}.
$$
\end{lemma}
To end the proof of Lemma \thmref{lem3.4}, it remains to
establish that for any
$x \in \B{R}$ and $y \in \B{R}_+$, 
$$
\chi(x) + \min \Bigl\{ \log(4), \frac{y}{8} \Bigr\}  
\leq \log \Bigl[ 1 + x + x^2/2 + \frac{a y}{2} \Bigr], 
$$
where 
\begin{equation}
\label{eq3.4}
a = \frac{3 \exp \bigl[ \sup(\chi) \bigr]}{4 \log(4)} 
= \frac{3 \exp(y_1 + 2 p_1^2)}{4 \log(4)} \leq 4.43. 
\end{equation}
Indeed $a$ should satisfy 
$$
a \geq \frac{2}{y} \biggl[ \exp \bigl[ \chi(x) \bigr] \min 
\Bigl\{ 4, \exp \Bigl( \frac{y}{8} \Bigr) \Bigr\} - \Bigl( 1 + x + \frac{x^2}{2} \Bigr) \biggr], \qquad x \in \B{R}, \quad y \in \B{R}_+.
$$
Since $\ds \exp \bigl[ \chi(x) \bigr] \leq 1 + x + \frac{x^2}{2}$, the right-hand side of this inequality is less than 
$$
\frac{2 \exp \bigl[ \chi(x) \bigr]}{y} \biggl( \min \Bigl\{ 4, 
\exp \Bigl( \frac{y}{8} \Bigr) \Bigr\}  - 1 \biggr).
$$
As $\ds y \mapsto y^{-1} \Bigl[ \exp \Bigl( \frac{y}{8} \Bigr) - 1 \Bigr]$ 
is increasing on $\B{R}_+$, this last expression reaches its maximum 
when $x \in \arg \max \chi$ and $\ds \exp \Bigl( \frac{y}{8} \Bigr) = 4$,
and is then equal to $\ds \frac{3 \exp \bigl[ \sup(\chi) \bigr]}{4 \log(4)}$, 
which is the value stated for $a$ in equation \eqref{eq3.4} above.  

\subsection{Proof of Proposition \thmref{prop6.1}}

Consider the function
$$
g(x) = 
\begin{cases} 
x - \log \bigl( 1 + x + x^2/2 \bigr), & x \geq 0,\\ 
x + \log \bigl( 1 - x + x^2/2 \bigr), & x \leq 0.
\end{cases}
$$
This function could also be defined as $g(x) = x - \psi(x)$, 
where $\psi$ is the wide version of the influence function 
defined by equation \myeq{eq1.2w}.
It is such that 
$$
g'(x) = \frac{x^2}{2\bigl( 1 + x + x^2/2 \bigr)} 
\leq \frac{x^2}{2}, \qquad x \geq 0.
$$
Therefore
$$
0 \leq g(x) \leq \frac{x^3}{6}, \qquad x \geq 0,
$$
implying by symmetry that 
$$
\lvert g(x) \rvert \leq \frac{\lvert x \rvert^3}{6}, \qquad x \in \B{R}.
$$
We can also remark that 
$$
g'(x) \leq \frac{x}{2(1 + \sqrt{2})}, \qquad x \geq 0,
$$
implying that
$$
\lvert g(x) \rvert \leq \frac{x^2}{4(1 + \sqrt{2})}, \qquad x \in \B{R}.
$$
As it is also obvious that $\lvert g(x) \rvert \leq \lvert x \rvert$, we get 
\begin{lemma}
$$
\lvert g(x) \rvert \leq \min \biggl\{ \lvert x \rvert, \frac{x^2}{4(1+\sqrt{2})}, \frac{\lvert x \rvert^3}{6} 
\biggr\}.
$$
\end{lemma}
Now let us write
$$
M = m + \frac{1}{\alpha n} \sum_{i=1}^n \psi
\bigl[ \alpha(Y_i - m) \bigr] + \frac{1}{\alpha n} 
\sum_{i=1}^n G_i,
$$
where $\ds G_i = g \bigl[ \alpha(Y_i -m) \bigr]$
and $\psi$ is the wide influence function 
of equation \myeq{eq1.2w}.
As we have already seen in Proposition \thmref{prop1.2}, with probability at least 
$1 - 2 \epsilon_1$, 
$$
\biggl\lvert \frac{1}{n\alpha} \sum_{i=1}^n 
\psi \bigl[ \alpha(Y_i - m) \bigr] \biggr\rvert
\leq \frac{\alpha v}{2} + \frac{\log(\epsilon_1^{-1})}{n
\alpha}. 
$$
On the other hand with probability at least $1 - \epsilon_2$
$$
\biggl\lvert \frac{1}{n \alpha} \sum_{i=1}^n G_i 
\biggr\rvert \leq \frac{1}{n\alpha} \sum_{i=1}^n 
\lvert G_i \rvert \leq 
\frac{\B{E} \bigl( \lvert G \rvert 
\bigr)}{\alpha} + \frac{\ds \B{E} \biggl[ \biggl( 
\sum_{i=1}^n \lvert G_i \rvert - \B{E} \bigl( \lvert
G \rvert \bigr) \biggr)^4 \biggr]^{1/4}}{n \alpha \epsilon_2^{1/4}}
$$
Let us put $H_i= \lvert G_i \rvert$ 
and let us compute
\begin{multline*}
\frac{1}{n} \B{E} \biggl[ \biggl( \sum_{i=1}^n 
H_i - \B{E} \bigl( H \bigr) \biggr)^4 \biggr] =  
\B{E}\Bigl\{ \bigl[ H - \B{E}(H) \bigr] ^4 \Bigr\} + 
3(n-1) \B{E} \Bigl\{ \bigl[ H - \B{E} (H) \bigr]^2 
\Bigr\}^2 \\ \shoveleft{ \qquad = 
\B{E} (H^4) - 4 \B{E}(H^3)\B{E}(H) + 6 \B{E}(H^2)
\B{E}(H)^2 - 3 \B{E}(H)^4} 
\\ \shoveright{+ 3(n-1) \bigl[ \B{E}(H^2)^2 - 2 \B{E}(H^2)\B{E}(H)^2 
+ \B{E}(H)^4 \bigr]} \\ \shoveleft{\qquad \leq \B{E}(H^4) + 2 \B{E}(H^2)
\B{E}(H)^2 - 3 \B{E}(H)^4 } 
\\ \shoveright{+ 3(n-1) \bigl[ \B{E}(H^2)^2 - 2 \B{E}(H^2)\B{E}(H)^2 
+ \B{E}(H)^4 \bigr] }
\\ \shoveleft{\qquad \leq \B{E}(H^4) + 3(n-1) \B{E}(H^2)^2 - (3n - 2) 
\B{E}(H)^4 } 
\\ \leq \alpha^4 \kappa v^2 + 3(n-1) 
\frac{\kappa^2 \alpha^8 v^4}{\bigl[ 4(1+\sqrt{2})\bigr]^4} - (3n-2) \B{E}(H)^4. 
\end{multline*}
Moreover
$$
\B{E}(H) \leq \frac{\alpha^3}{6} \B{E} \bigl( 
\lvert Y - m \rvert^3 \bigr) \leq \frac{\alpha^3}{6} 
\sqrt{\kappa v^3} 
$$
Thus with probability at least $1 - 2 \epsilon_1 - \epsilon_2$, 
\begin{multline*}
\lvert M - m \rvert \leq \frac{\alpha v}{2} 
+ \frac{\log(\epsilon_1^{-1})}{n\alpha} \\ + \sup_{x 
\in (0, \sqrt{\kappa \alpha^6 v^3}/6)} 
\frac{x}{\alpha} + n^{-3/4}\alpha^{-1} \epsilon_2^{-1/4} 
\biggl[ \kappa \alpha^4 v^2 + \frac{3(n-1) \kappa^2 
\alpha^8 v^4}{\bigl[ 4(1 + \sqrt{2})\bigr]^4} 
- (3n-2) x^4 \biggr]^{1/4} 
\\ \leq \frac{\alpha v}{2} + \frac{\log(\epsilon_1^{-1})}{n\alpha} 
+ \frac{\sqrt{\kappa v^3}\alpha^2}{6} 
+ \frac{\sqrt{v}}{n^{3/4}} \left( \frac{\kappa}{\epsilon_2}\right)^{1/4} \biggl[ 1 + 
\frac{3(n-1) \kappa \alpha^4 v^2}{\bigl[ 4(1 + \sqrt{2})\bigr]^4} \biggr]^{1/4}. 
\end{multline*}
Let us take
$$
\alpha = \sqrt{\frac{2 \log(\epsilon_1^{-1})}{nv}}
$$
and let us put $\epsilon_1 = \lambda \epsilon$
and $\epsilon_2 = (1 - \lambda) 2 \epsilon$.

The bound can either be optimized in $\lambda$ or we can 
for simplicity choose $\lambda$ to balance the following 
factors 
$$
\sqrt{\frac{2 v \log(\lambda^{-1} \epsilon^{-1})}{n}} 
\simeq \frac{\sqrt{v}}{n^{3/4}} \biggl(\frac{\kappa}{2\epsilon} \biggr)^{1/4} \frac{\lambda}{4}.
$$
This leads to consider the value
$$
\lambda = \min \biggl\{ \frac{1}{2}, 
2 \sqrt{2 \log\left(\frac{\kappa}{2n\epsilon^{5}}\right)} 
\left( \frac{2 n \epsilon}{\kappa}
\right)^{1/4} \biggr\}. 
$$
As stated in Proposition \thmref{prop6.1}, 
with probability at least $1 - 2 \epsilon$, 
\begin{multline*}
\lvert M - m \rvert \leq \sqrt{\frac{2v\log(\lambda^{-1}
\epsilon^{-1})}{n}} + \frac{\sqrt{\kappa v} \log(
\lambda^{-1}\epsilon^{-1}) }{3n}
\\ + \sqrt{\frac{v}{n}} \left(\frac{\kappa}{
2 (1-\lambda) n \epsilon} \right)^{1/4} \Biggl[ 
1 + \frac{3(n-1)\kappa \log(\lambda^{-1}
\epsilon^{-1})^2 }{\bigl[ 4^3(1 + \sqrt{2})^4  
n^2} \Biggr]^{1/4} 
\\ \underset{\epsilon \rightarrow 0}{\simeq}
\sqrt{\frac{v}{n}} \left( \frac{\kappa}{2 n \epsilon} \right)^{1/4}. 
\end{multline*}

Another bound can be obtained applying Chebyshev's
inequality directly to the fourth moment of the empirical 
mean, which however does not reach the right speed when $\epsilon$ is small 
and $n$ large.

\begin{prop}
\label{prop7.3.2}
For any probability distribution whose kurtosis is not greater than $\kappa$, 
the empirical mean $M$ is such that with probability at least $1 - 
2 \epsilon$, 
$$
\lvert M - m \rvert \leq \biggl( \frac{3(n-1) + \kappa}{2 n \epsilon} 
\biggr)^{1/4} \sqrt{\frac{v}{n}}
$$
\end{prop}

\begin{proof} Let us assume to simplify notations and without
loss of generality that $\B{E}(Y) = 0$. 
$$
\B{E} \bigl( M^4 \bigr) = 
\frac{1}{n^4} \sum_{i=1}^n \B{E} (Y_i^4) 
+ \frac{1}{n^4} \sum_{i < j} 6 \B{E}(Y_i^2)\B{E}(Y_j^2) 
= \frac{\B{E}(Y^4)}{n^3} + \frac{3 (n-1) \B{E}(Y^2)^2}{n^3}.
$$
It implies that 
$$
\B{P} \Bigl( \lvert M - m \rvert \geq \eta \Bigr) \leq \frac{\B{E}(M^4)}{\eta^4} 
\leq \frac{\bigl[3(n-1) + \kappa \bigr] v^2}{n^3 \eta^4},
$$
and the result is proved by considering 
$\ds 2 \epsilon = \frac{\bigl[ 3(n-1)+\kappa \bigr] v^2}{n^3 \eta^4}$.  
\end{proof}

In our comparisons with new estimators, we took the minimum over the three bounds
given by Proposition \thmref{prop6.1}, Proposition  \thmref{prop7.3.2} and equation \myeq{eq7.1}.

\subsection{Proof of Proposition \thmref{prop2.1.2}}

Let us consider the distributions $\B{P}_1$ and 
$\B{P}_2$ of the sample $(Y_i)_{i=1}^n$ obtained when 
the marginal distributions are respectively the Gaussian
measure with variance $v$ and mean $m_1 = -\eta$ and the 
Gaussian measure with variance $v$ and mean $m_2 = \eta$. 
We see that, whatever the estimator $\wh{\theta}$, 
\begin{multline*}
\B{P}_1(\wh{\theta} \geq m_1 + \eta) 
+ \B{P}_2(\wh{\theta} \leq m_2 - \eta) = 
\B{P}_1(\wh{\theta} \geq 0) + \B{P}_2(\wh{\theta} \leq 0)
\\ \geq (\B{P}_1 \wedge \B{P}_2)(\wh{\theta} \geq 0) 
+ (\B{P}_1 \wedge \B{P}_2) (\wh{\theta} \leq 0) 
\geq \lvert \B{P}_1 \wedge \B{P}_2 \rvert,
\end{multline*}
where $\B{P}_1 \wedge \B{P}_2$ is the measure whose density
with respect to the Lebesgue measure (or equivalently 
with respect to any dominating measure, such as $\B{P}_1 + 
\B{P}_2$) is the minimum of the 
densities of $\B{P}_1$ and $\B{P}_2$ and whose total variation 
is $\lvert \B{P}_1 \wedge \B{P}_2 \rvert$. 

Now, using the fact that the empirical mean is a sufficient 
statistics of the Gaussian shift model, it is easy to 
realize that 
$$
\lvert \B{P}_1 \wedge \B{P}_2 \rvert = \B{P}_1(M \geq m_1 + \eta) 
+ \B{P}_2(M \leq m_2 - \eta),
$$
which obviously proves the proposition. 

\subsection{Proof of Proposition \thmref{prop2.2}}

Let us consider the distribution with support $ \{ 
- n \eta, 0, n \eta \}$ defined by
$$
\B{P}\bigl( \{n \eta\} \bigr) = \B{P}\bigl( \{-n\eta \}\bigr) =  
\bigl[ 1 - \B{P}(\{0\}) \bigr]/2 = \frac{v}{2n^2 \eta^2}.
$$
It satisfies $\B{E}(Y) = 0$, $\B{E}(Y^2) = v$ and 
$$
\B{P}(M \geq \eta) = \B{P}(M \leq - \eta) 
\geq \B{P}(M = \eta) = \frac{v}{2n\eta^2}
\left(1 - \frac{v}{n^2\eta^2} \right)^{n-1}.
$$

\subsection{Proof of Proposition \thmref{prop7.3}}

Let us consider for $Y$ the following distribution, 
with support $\{-n\eta, - \xi, \xi, n \eta \}$, where $\xi$ and $\eta$ are two positive real parameters, 
to be adjusted to obtain the desired variance and kurtosis.

\begin{align*}
\B{P}(Y = - n \eta) & = \B{P}(Y = n \eta) = q,\\
\B{P}(Y = - \xi) & = \B{P}(Y = \xi ) = \frac{1}{2} - q.
\end{align*}

In this case
\begin{align*}
m & = 0,\\
v & = (1 - 2q) \xi^2 + 2 q n^2 \eta^2, \\ 
\kappa v^2 & = (1 - 2q) \xi^4 + 2 q n^4 \eta^4.
\end{align*}
Thus 
$$
\kappa = f_q \bigl( n \eta / \xi \bigr),
$$
where
$$
f_q(x) = \frac{1 - 2q + 2q x^4}{\bigl( 1 - 2q 
+ 2q x^2 \bigr)^2}, \qquad x \geq 1.
$$
It is easy to see that $f_q$ is an increasing one
to one mapping of $(1, + \infty($ into itself.
We obtain
\begin{align*}
\xi & = \sqrt{ \frac{v}{1 - 2q + 2q 
\bigl[ f_q^{-1}(\kappa) \bigr]^2}} \leq \sqrt{v},
\\ \text{and } \quad \eta & = 
\frac{\xi f_q^{-1}(\kappa)}{n}.
\end{align*}
Consequently 
$$
\left( \frac{\kappa}{2q} \right)^{1/4} \frac{\sqrt{v}}{n} \geq \eta \geq \left( \frac{\kappa - 1 + 2q }{2q} \right)^{1/4} \frac{\sqrt{v}}{n}. 
$$
On the other hand, 
\begin{multline*}
\B{P} \bigl( M \geq \eta - \gamma \bigr) 
\geq \sum_{i=1}^n \B{P} \biggl( Y_i = n \eta, 
\frac{1}{n} \sum_{j, j\neq i} Y_j \geq - \gamma,
Y_j \in \{ - \xi, \xi \}, j \neq i \biggr) \\
\shoveleft{ \qquad = n \B{P}(Y_1 = n\eta ) (1 - 2q)^{n-1}} \\
\times \biggl[ 1 - \B{P} 
\biggl( \frac{1}{n} \sum_{i=2}^n Y_i \geq \gamma
\; \Big\vert \; Y_j \in \{ - \xi, \xi \bigr\}, 2 \leq 
j \leq n \biggr) \Biggr]. 
\end{multline*}
Let us remark that
\begin{multline*}
\B{P} \biggl( \frac{1}{n} \sum_{i=2}^n Y_i 
\geq \gamma \; \Big\vert \; Y_j \in \{-\xi, \xi \}, 
2 \leq j \leq n \biggr) \\ \leq \inf_{\lambda \geq 
0} \cosh \bigl( 
\lambda \xi / n \bigr)^{n-1}
\exp( -  \lambda \gamma) \leq \inf_{\lambda \geq 0} 
\exp \biggl( 
\frac{\lambda^2 \xi^2}{2 n} - \lambda \gamma
\biggr) \\ = \exp \biggl( - \frac{n \gamma^2}{2 \xi^2}
\biggr) \leq \exp \left( - \frac{n \gamma^2}{2 v} \right).
\end{multline*}
Also, by symmetry
$$
\B{P} \biggl( \frac{1}{n} \sum_{i=2}^n Y_i 
\geq \gamma \; \Big\vert \; Y_j \in \{-\xi, \xi \}, 
2 \leq j \leq n \biggr) \leq \frac{1}{2}. 
$$
Thus, as $m = 0$,
\begin{align*}
\B{P} \bigl( \lvert M - m \rvert \geq \eta - \gamma \bigr) 
& \geq 2 n q \bigl( 1 - 2q \bigr)^{n-1} 
\max \biggl\{ \frac{1}{2}, 1 - \exp \left( - \frac{n 
\gamma^2}{2 v} \right) \biggr\}.
\end{align*}
Let us put 
$$
\chi = \min \biggl\{ \frac{1}{2}, \exp \left( - \frac{n \gamma^2}{2 v} \right) 
\biggr\}
$$
and let us assume that $\ds \epsilon \leq \frac{1}{16}$.
Let us put
$$
q = \frac{\epsilon}{n (1 - \chi)} \left( 
1 - \frac{4 \epsilon}{n(1-\chi)} \right)^{-(n-1)} \leq 
\frac{2 \epsilon}{n(1-\chi)},
$$
the last inequality being a consequence of the convexity of $x 
\mapsto (1 - x)^{n-1} \geq 1 - (n-1) x \geq 1 - n x, 0 \leq x \leq 1$. 
Let us remark then that
$$
\B{P} \bigl( \lvert M - m \rvert \geq \eta - \gamma \bigr) \geq  
\frac{ 2 \epsilon \bigl(1 - 2 q\bigr)^{n-1}}{\ds \left( 1 
- \frac{4 \epsilon}{n (1 - \chi)} \right)^{n-1}}
\geq 2 \epsilon.
$$
Thus, with probability at least $2 \epsilon$, 
\begin{multline*}
\lvert M - m \rvert \geq \left[ \frac{(\kappa - 1 
+ 2 \epsilon / n )(1 - \chi)}{2 n \epsilon } \right]^{1/4} \left( 
1 - \frac{4 \epsilon}{n(1-\chi)} \right)^{(n-1)/4} \sqrt{\frac{v}{n}} \\ - \sqrt{\frac{2 \log(\chi^{-1}) v}{n}} \\ 
\geq \sup_{0 < \chi \leq 1/2} \left[ \frac{(\kappa - 1 + 2 \epsilon/n) \bigl(1 - \chi - 4 \epsilon 
\bigr)}{2n\epsilon} \right]^{1/4}\sqrt{\frac{v}{n}} - \sqrt{ \frac{2 \log(\chi^{-1}) \B{1}\bigl(\chi < 1/2\bigr) v}{n}}.
\end{multline*}
In order to obtain Proposition \thmref{prop7.3}, it is enough to restrict
optimization with respect to $\chi$ to the two values
$$
\chi = \frac{\bigl( n \epsilon \bigr)^{1/4}}{2} \text{ and } \chi = \frac{1}{2}.
$$

\section{Generalization to non-identically distributed samples}

The assumption that the sample is identically distributed can 
be dropped in Proposition \thmref{prop1.4}. Indeed, assuming only that the random variables
$(Y_i)_{i=1}^n$ are independent, meaning that their 
joint distribution is of the product form $\bigotimes_{i=1}^n 
\B{P}_i$, we can still write, 
for $W_i = \pm \alpha (Y_i - \theta)$,
\begin{multline*}
\B{E} \biggl\{ \exp \biggl[ \sum_{i=1}^n \log \biggl( 
1 + W_i + \frac{W_i^2}{2} 
\biggr) \biggr] \biggr\}\\ 
= \exp \biggl\{ \sum_{i=1}^n \log \biggl[ 1 + \B{E} \bigl( W_i) 
+ \frac{ \B{E}\bigl( W_i^2 \bigr)}{2} \biggr] \biggr\} 
\\ \leq  \exp \biggl\{ n \log \biggl[ 1 + \frac{1}{n} \sum_{i=1}^n 
\B{E} \bigl( W_i \bigr) + \frac{1}{2n} \sum_{i=1}^n 
\B{E} \bigl( W_i^2 \bigr) \biggr] \biggr\}.
\end{multline*}

Starting from these exponential inequalities, we can reach the same 
conclusions as in Proposition \thmref{prop1.4}, as long as we set
\begin{align*}
m & = \frac{1}{n} \sum_{i=1}^n \B{E}(Y_i),\\
\text{and } \quad v & = \frac{1}{n} \sum_{i=1}^n \B{E} \bigl[ (Y_i - m)^2 \bigr].
\end{align*}
We see that the mean marginal sample distribution 
$\ds \frac{1}{n} \sum_{i=1}^n \B{P}_i$ is playing here the same 
role as the marginal sample distribution in the i.i.d. case.

\section{Experiments}

\subsection{Mean estimators}
Theoretical bounds can explore high confidence levels better
than experiments, and have also the advantage to hold true 
in the worst case. They have led us to introduce new M-estimators, 
and in particular the one described in Proposition \thmref{prop3.3}.
Nonetheless, they may be expected to be 
rather pessimistic and are clearly insufficient to explore
the moderate deviations of the proposed estimators. In particular
it would be interesting to know whether the improvement in 
high confidence deviations has been obtained as a trade-off
between large and moderate deviations (by which we mean 
a trade-off between the left part and the tail 
of the quantile function of $\lvert \wh{\theta} 
- m \rvert$).

This is a good reason to launch into some experiments. 
We are going to test sample distributions of the form
$$
\sum_{i=1}^d p_i \C{N}(m_i, \sigma_i^2),
$$
where $d \in \{1,2,3\}$, $(p_i)_{i=1, \dots, d}$ is a vector 
of probabilities and $\C{N}(m, \sigma^2)$ is as usual 
the Gaussian measure with mean $m$ and variance $\sigma^2$. 
To visualize results, we have chosen to plot the quantile 
function of the deviations from the true parameter, 
that is the quantile function of $\lvert \wh{\theta} - m \rvert$, 
where $\wh{\theta}$ is one of the mean estimators studied 
in this paper.

Let us start with an asymmetric distribution with a so to speak 
intermittent high variance component. Let us take accordingly
\begin{align*}
p_1 & = 0.7, & m_1 & = 2, & \sigma_1 & = 1,\\
p_2 & = 0.2, & m_2 & = -2, & \sigma_2 & = 1,\\
p_3 & = 0.1, & m_3 & = 0, & \sigma_3 & = 30.
\end{align*}
In this case, $m=1$, $\kappa = 27.86$ and $v = 93.5$, so that, when 
the sample size $n$ is in the range $100 \leq n \leq 1000$, the variance
estimates we are proposing in this paper are not proved to be valid. 
For this reason we will challenge the empirical mean with the two 
following estimates : the estimate $\wh{\theta}_{\alpha}$ of Proposition 
\thmref{prop1.4} (where $\alpha$ is chosen using the true value 
of $v$) and a naive plug-in estimate, $\wh{\theta}_{\wh{\alpha}}$, 
where $\wh{\alpha}$ is set as was $\alpha$, replacing the true 
variance $v$ with its unbiased estimate $\wh{V}$ given by 
equation \myeq{eq4.1bis}. The parameter $\epsilon$ is set to 
$\epsilon = 0.05$ for both estimators, targeting the probability 
level $1 - 2 \epsilon = 0.9$.

We will plot also the sample median estimator, in this case where 
the distribution median is different from its mean, 
to show that robust location estimators for symmetric
distributions do not apply here.

When the sample size is $n = 100$, we obtain the following results
(computed from $1000$ experiments).\\
\mbox{} \hfill \raisebox{-1cm}[8.8cm][0cm]{\includegraphics{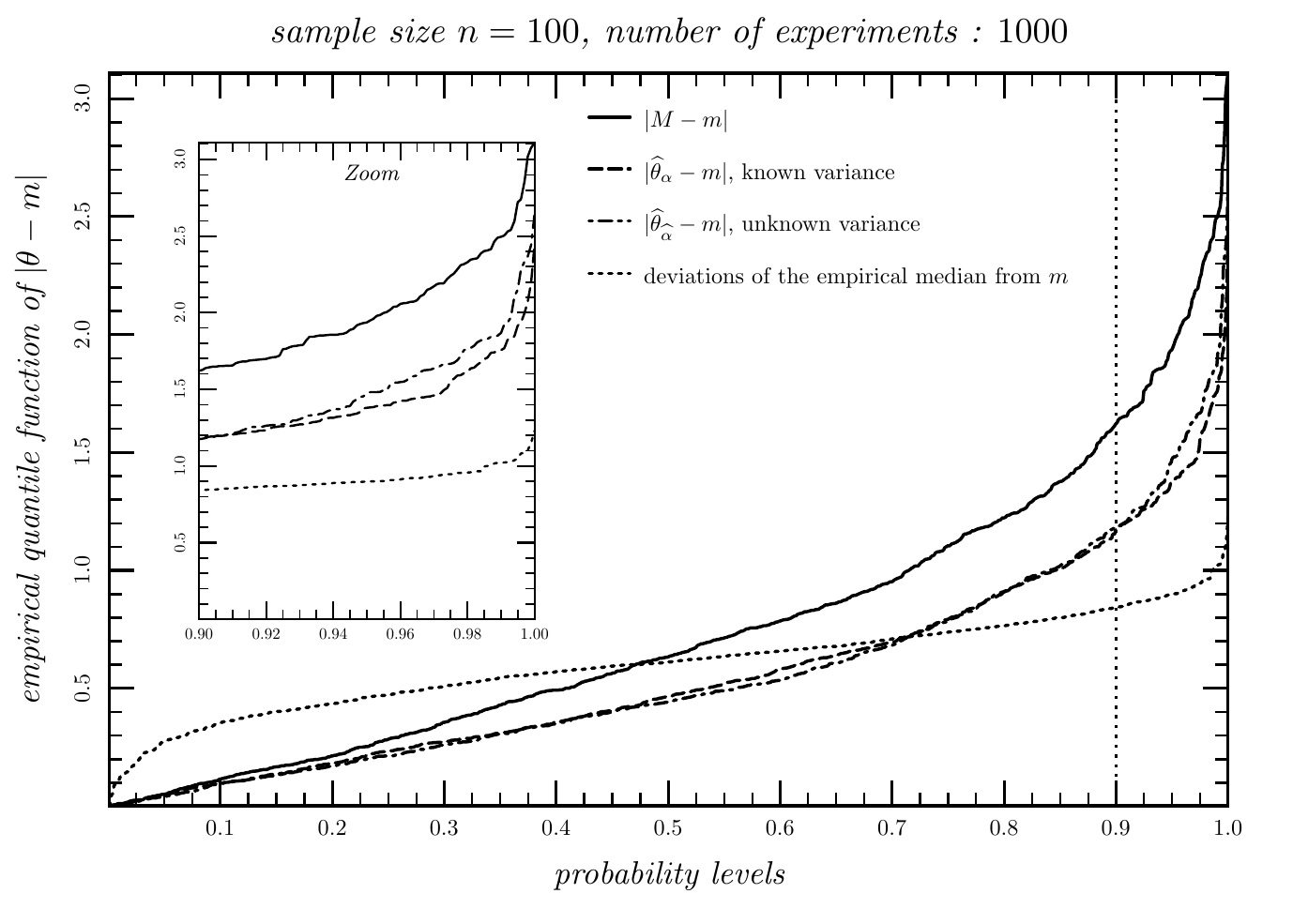}} 
\hfill \mbox{}\\
In this first example, the new M-estimators have uniformly better quantiles 
than the empirical mean, at any probability level. 
Moreover the variance can be harmlessly estimated 
from the data when it is unknown.
Thus, in this case, the empirical 
mean is outperformed from any conceivable point of view.

\vfill 

Let us now increase the sample size to $n=1000$.\\
\mbox{} \hfill \raisebox{-0.4cm}[10cm][0cm]{\includegraphics{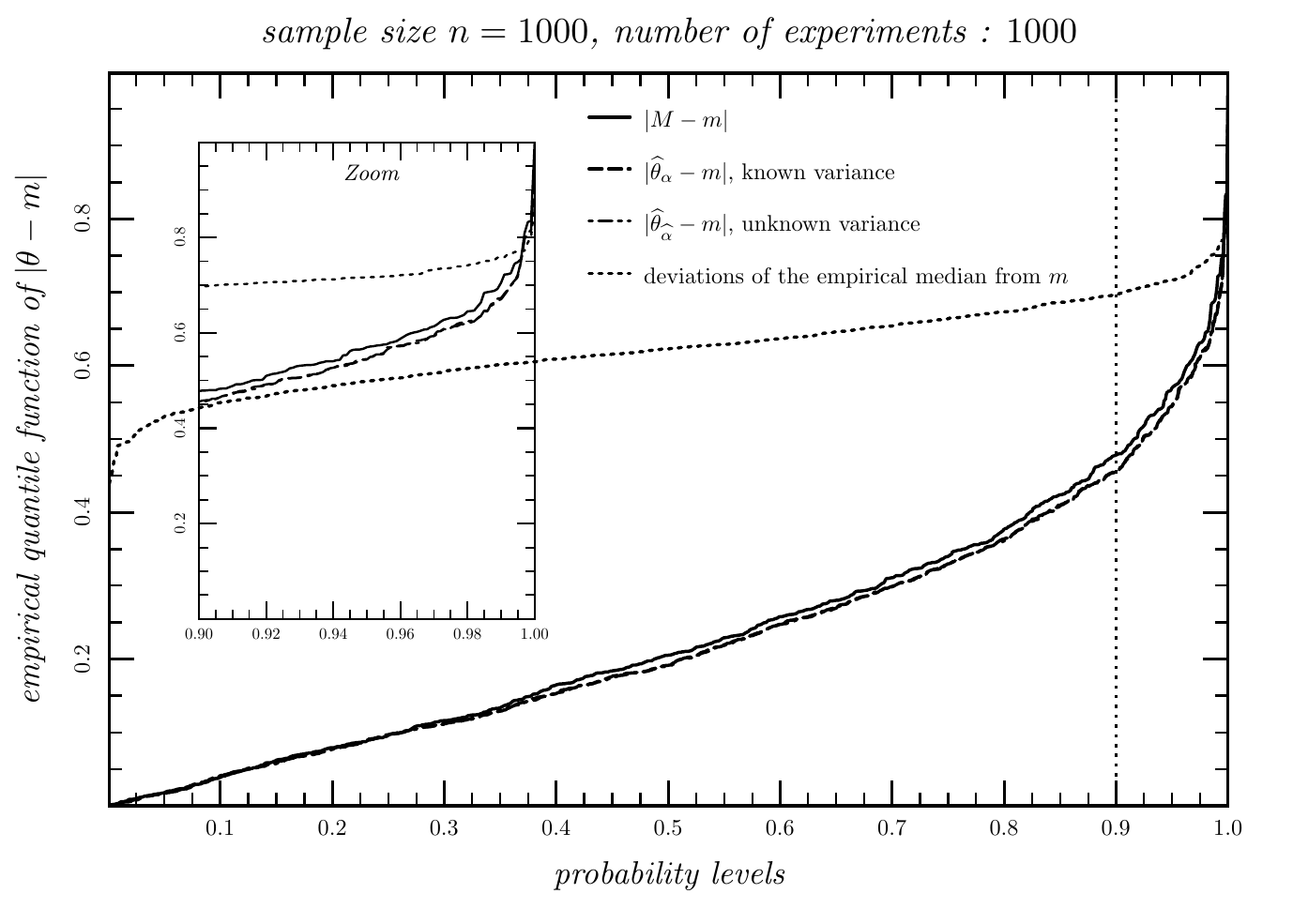}}   
\hfill \mbox{}\\
As should be expected, the values of the three estimators 
get close for this larger sample size (whereas it becomes 
more obvious that the empirical median is estimating 
something different from the mean). 

\vfill

The empirical mean can still be challenged 
for this larger sample size, but for a different 
sample distribution.
To illustrate this, 
let us consider a situation as simple as  
the mixture of two centered Gaussian measures. 
Let $d = 2$ and 
\begin{align*}
p_1 & = 0.99, & m_1 & = 0, & \sigma_1 & = 1, \\
p_2 & = 0.01, & m_2 & = 0, & \sigma_2 & = 30.
\end{align*} 
\pagebreak

Here $\kappa = 243.5$, $v = 9.99$ and $m=0$. 
We take $\epsilon = 0.005$, targeting the probability level 
$1 - 2 \epsilon = 1 - 10/n = 0.99$.\\
\mbox{} \hfill \raisebox{-0.8cm}[9.2cm][0cm]{\includegraphics{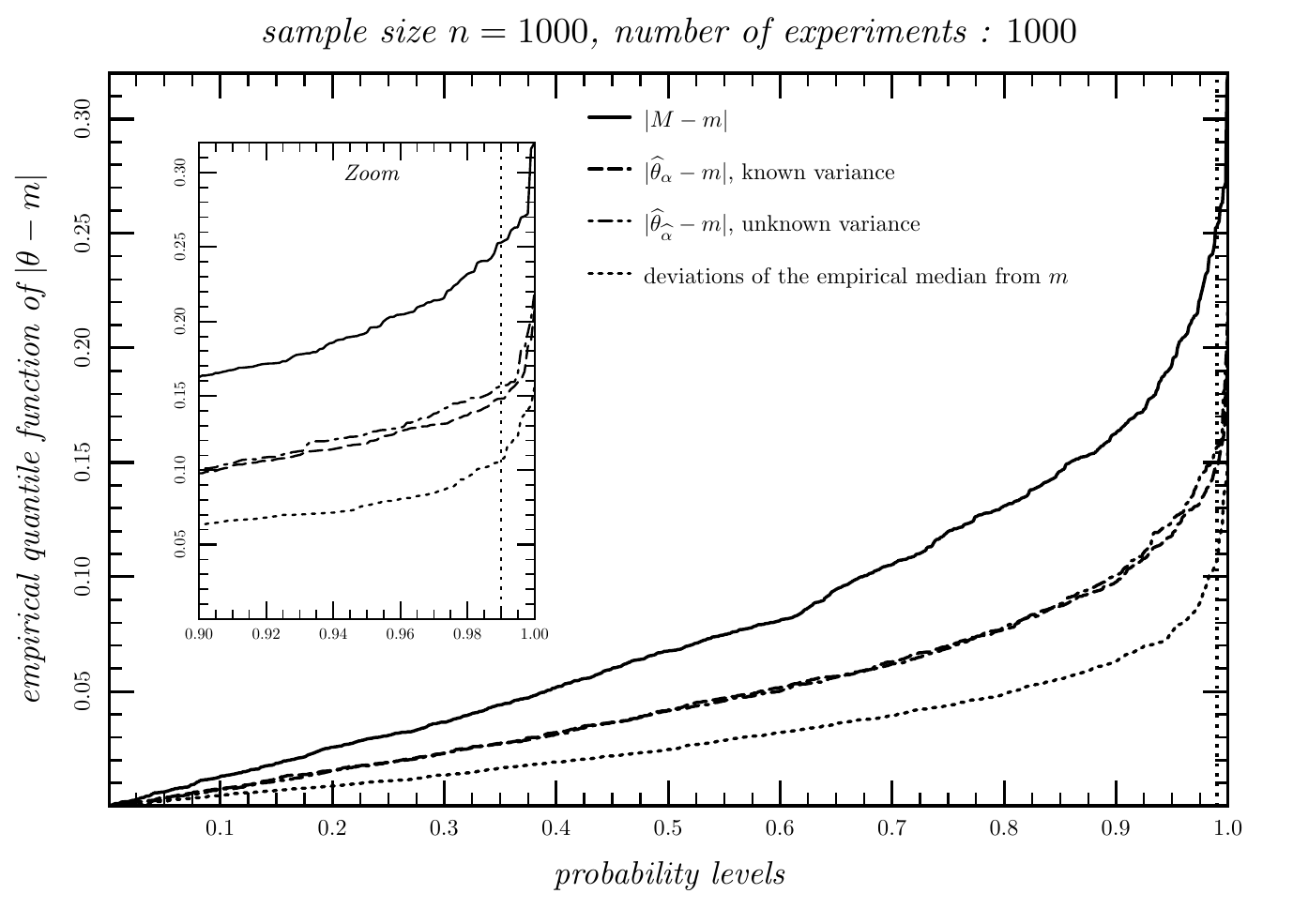}}
\hfill \mbox{}

\vfill

Let us show some heavily asymmetric situation where the left-hand side 
of the quantile function of the new estimators does not 
improve on the empirical mean.
In what follows $\kappa = 33.4$, $v = 72.25$, $m = -1.3$, 
and the mixture parameters are
\begin{align*}
p_1 & =  0.94 & m_1 & = 0, & \sigma_1 & = 1, \\
p_2 & = 0.01, & m_2 & = 20, & s_2 & = 20, \\ 
p_3 & = 0.05, & m_3 & = -30, & s_3 & = 20.
\end{align*}
We plot below two estimators using the value of the variance,
optimized for the confidence level $1 - 2 \epsilon$ with $\epsilon = 
0.05$ and $\epsilon = 0.0005$ respectively (the estimators 
with unknown variance show the same behaviour).

Here, choosing a moderate value of the estimator parameter $\epsilon$
is required to improve uniformly on the empirical mean performance, whereas 
higher values of $\epsilon$ produce a trade-off between low and high probability levels.
Whether this remains true in general would require
to be confirmed by more extensive experiments.

\noindent \mbox{} \hfill \raisebox{-0.8cm}[9.2cm][0cm]{\includegraphics{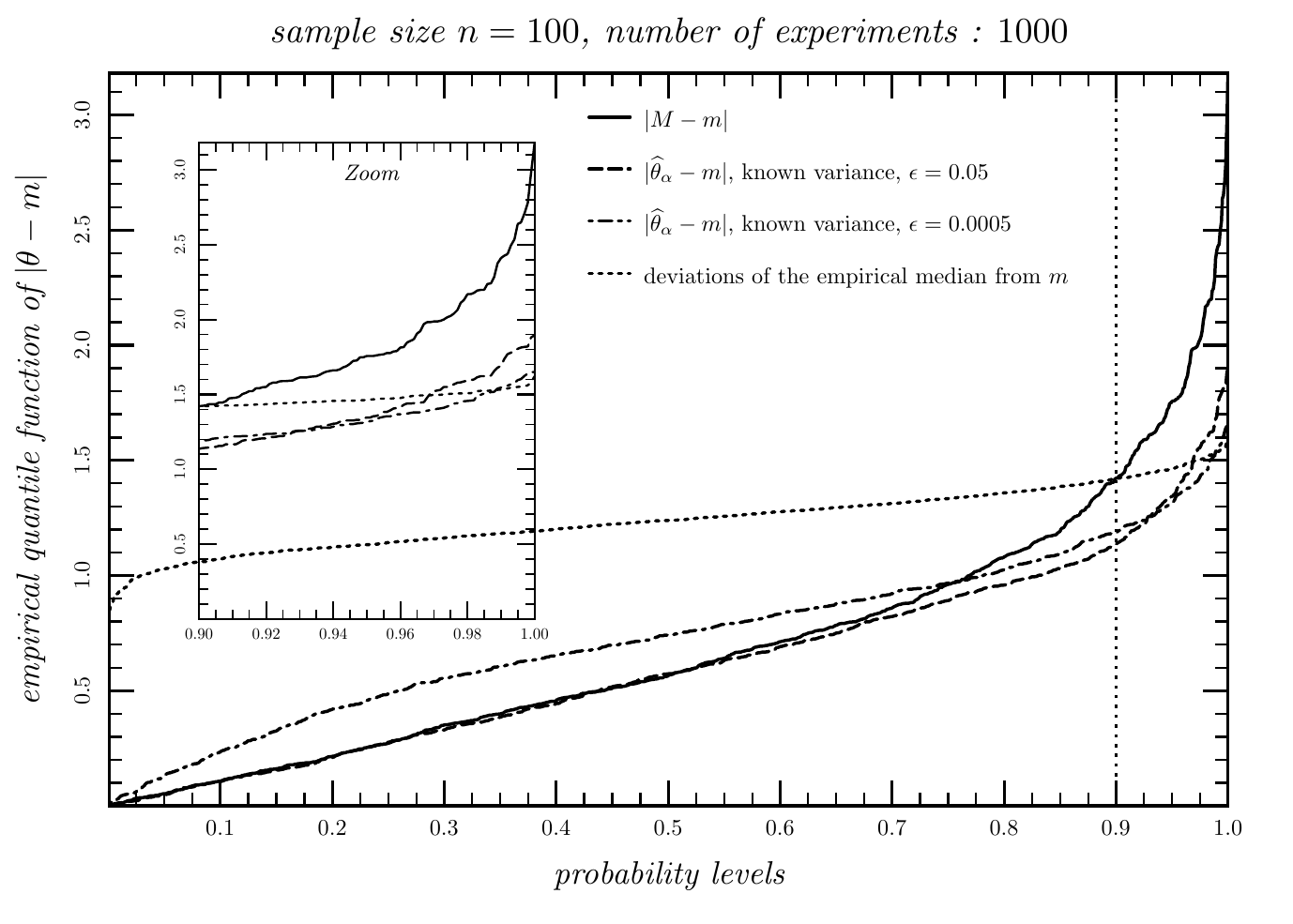}} \hfill \mbox{}

Let us end this section with a Gaussian sample.\\
\mbox{} \hfill \raisebox{-0.8cm}[9.2cm][0cm]{\includegraphics{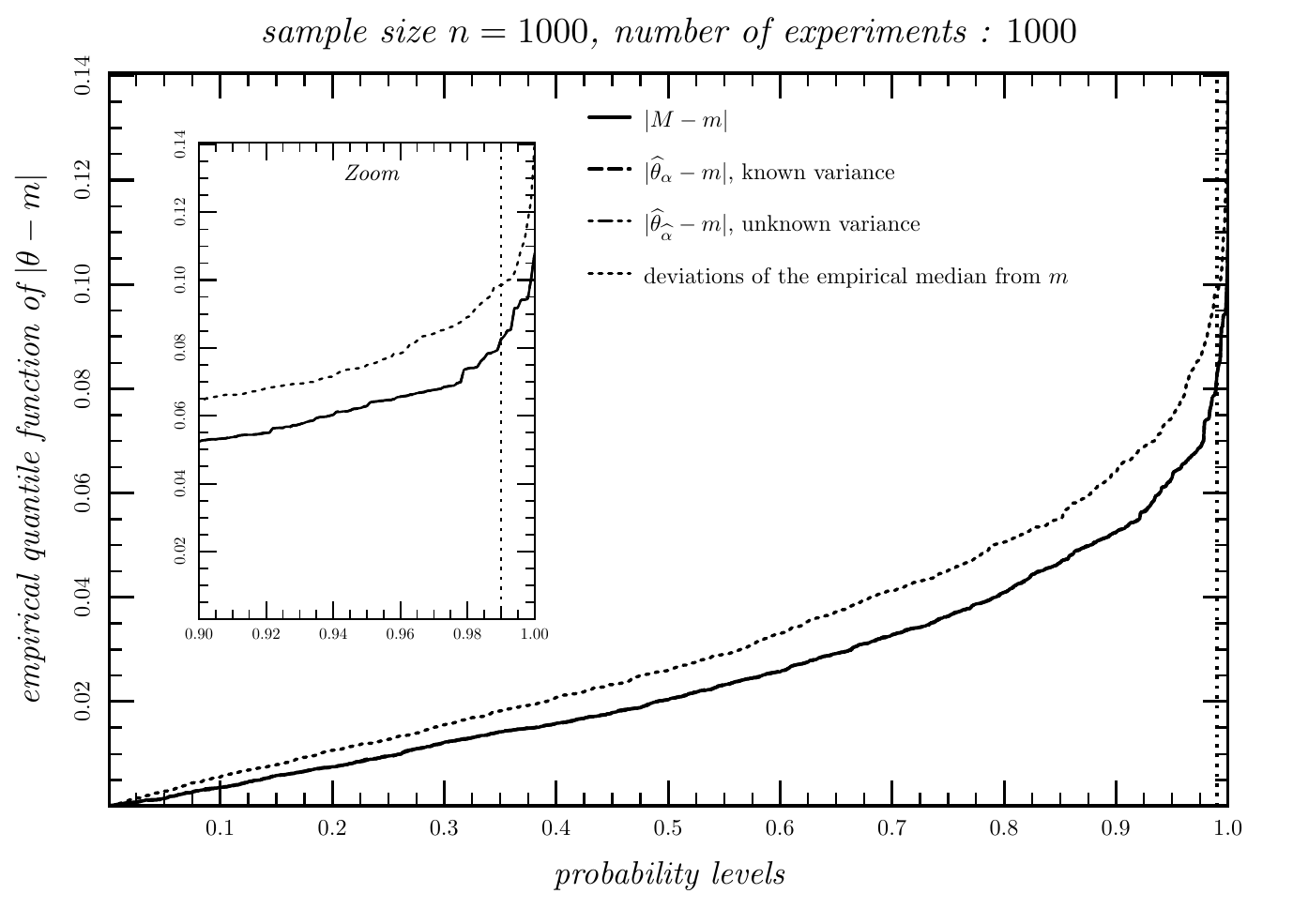}} \hfill \mbox{}

When the sample is Gaussian, as could be expected, our new M-estimators 
coincide with the empirical mean. What we obtained  
for a sample size $n = 1000$ could also be observed for other 
sample sizes. The deviations of 
the empirical median are higher in the Gaussian case, 
as proved in Proposition \thmref{prop2.1.2}
(stating that the deviations of the empirical mean of a Gaussian 
sample are optimal).

\subsection{Variance estimators}

Let us now test some variance estimates. This is an example 
where the usual unbiased estimate $\wh{V}$ defined by 
equation \myeq{eq4.1bis} shows its weakness. 
To demonstrate things on simple sample distributions, 
we choose again a mixture of two Gaussian measures
\begin{align*}
p_1 & = 0.995, & m_1 & = 0, & \sigma_1 & = 1, \\
p_2 & = 0.005, & m_2 & = 1, & \sigma_2 & = 5.
\end{align*}
Here $\kappa = 10.357$, $v = 1.125$ and $m = 0.005$.
To be in a situation where the variance estimates of
Proposition \thmref{prop4.1} work at high confidence
levels, we choose a sample size $n = 2000$,
and use in the estimator the parameters
$\kappa_{\max}  = 6*n/1000 = 12$, 
$p=2$ and $\epsilon = 0.0025$ (targeting the probability level 
$1 - 2 \epsilon = 1 - 10/n = 0.0995$).\\
\mbox{} \hfill \includegraphics{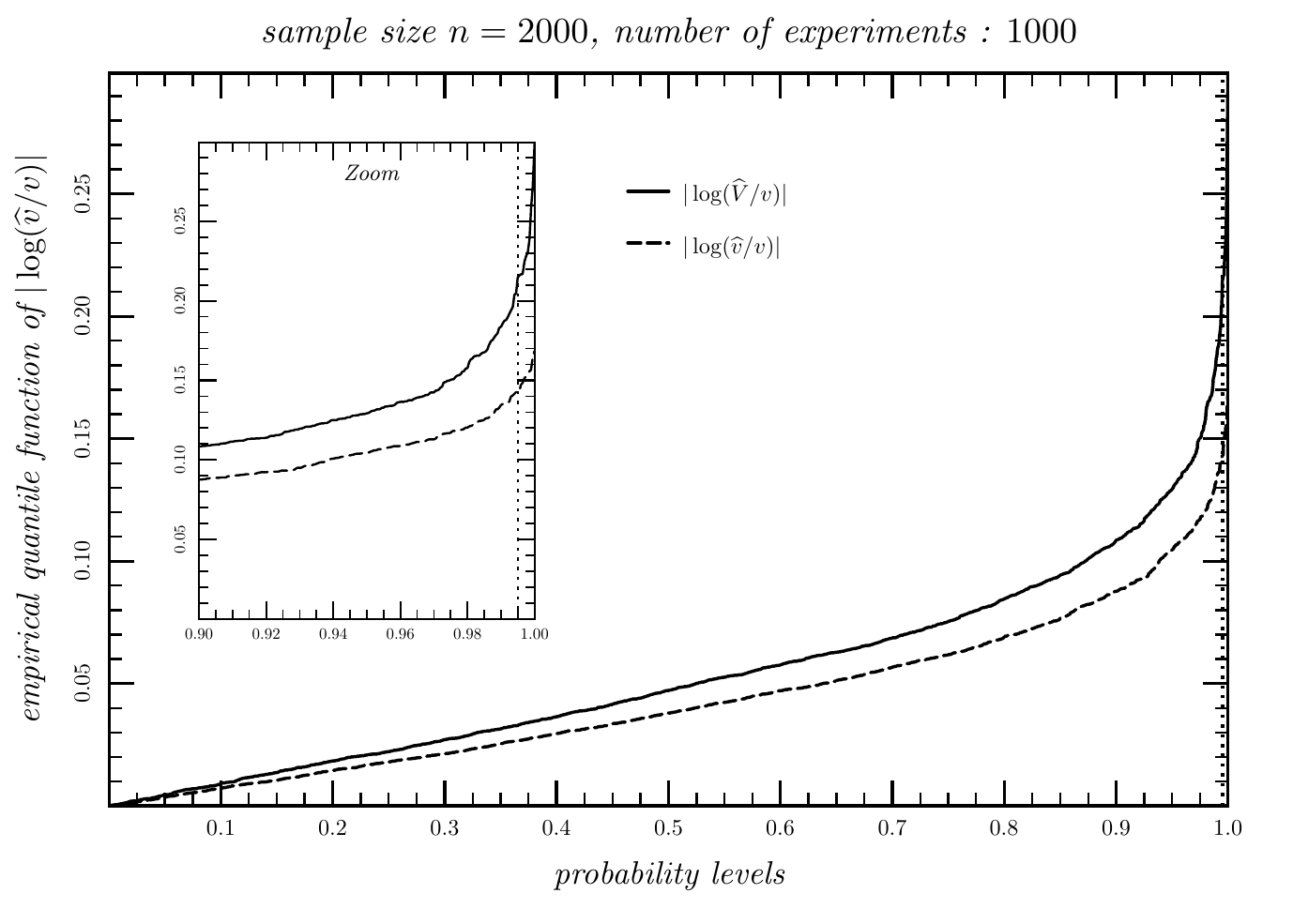}   
\hfill \mbox{}\\[-4ex]
So, for the variance as well as for the mean, there are
simple situations in which our new estimators 
perform better in practice than the more traditional 
ones. 

\subsection{Computation details}

To compute the estimators in these experiments, we used 
the two following iterative schemes (performing 
two iterations was enough for all the examples shown). 
\begin{align*}
\theta_0 & = M, &  \theta_{k+1} & = r(\theta_k) + \theta_k,\\
\beta_0 & = \frac{\delta - y}{\wh{V}}, & \beta_{k+1} 
& = \frac{\delta - y}{Q(\beta_k) + \delta} \beta_k.
\end{align*}
They are based on two principles: they have 
the desired fixed point and their right-hand side 
would be independent respectively of $\theta_k$
and $\beta_k$ if $\psi$ was replaced with the 
identity. The fact that $\psi$ is close to the identity
explains why the convergence is fast and only a few 
iterations are required. 

These numerical schemes involve only a reasonable amount
of computations,
opening the 
possibility to use the new estimators
in improved filtering algorithms in signal 
and image processing (a subject for future 
research that will not be pushed further
in this paper).

\section{Conclusion}

Theoretical results show that, for some sample 
distributions, the deviations of the empirical mean
at confidence levels higher than $90\%$ are larger than 
the deviations of some well chosen M-estimator.
Moreover, in our experiments, based on non-Gaussian sample 
distributions, the deviation quantile 
function of this M-estimator is uniformly below the quantile 
function of the empirical mean. The improvement
of the confidence interval at level $90\%$ can be more 
than $25\%$. 

Using Lepski's adapting approach offers a response with proved 
properties to the case when the variance is unknown.
For sample sizes starting from $1000$, an alternative 
is to use an M-estimator of the variance depending 
on some assumption on the value of the kurtosis.
However, it seems that the variance can in practice be 
estimated by the usual unbiased estimator $\wh{V}$,
defined by equation \myeq{eq4.1bis}, and plugged in
the estimator of Proposition \thmref{prop1.4}, 
although there is no mathematical warrant for this 
simplified scheme. 

\nocite{AuCat10b}
\bibliographystyle{plain}
\bibliography{ref2}

\end{document}